\documentclass[11pt]{article}
\usepackage[marginparwidth=2cm,marginparsep=-0.7cm]{geometry}
\usepackage{graphicx}
\usepackage{amsmath,amssymb,amsthm,bbm}
\usepackage[utf8]{inputenc}

\usepackage[english]{babel}
\usepackage{enumerate}
\usepackage{hyperref}

\usepackage[margin=1in,font=small]{caption}

\usepackage[active]{srcltx}
\numberwithin{equation}{section}
\usepackage[T1]{fontenc}
\usepackage{color,soul}
\usepackage[dvipsnames]{xcolor}
\usepackage[cm]{fullpage}

\newtheorem{theo}{Theorem}[section]
\newtheorem{lem}[theo]{Lemma}
\newtheorem{defi}[theo]{Definition}
\newtheorem{cor}[theo]{Corollary}
\newtheorem{prop}[theo]{Proposition}
\newtheorem{rmk}[theo]{Remark}
\newtheorem{assump}{Assumption}
\newcommand{\eps}{\varepsilon}

\newcommand{\N}{\mathbb{N}}
\newcommand{\R}{\mathbb{R}}
\newcommand{\Z}{\mathbb{Z}}

\newcommand{\Sph}{S^{N-1}}
\newcommand{\ol}{\overline}
\renewcommand{\ul}{\underline}
\newcommand{\vp}{\varphi}
\renewcommand{\.}{\cdot}

\def\1{\mathbbm{1}}

\DeclareMathOperator{\dv}{div}

\DeclareMathOperator\supp{supp}

\DeclareMathOperator\inter{int}
\DeclareMathOperator\Conv{Conv}

\let\mc=\mathcal
\let\t=\widetilde

\def\ptf{pulsating travelling front}

\newenvironment{formula}[1]{\begin{equation}\label{#1}}
{\end{equation}\noindent}

\def\Fi#1{\begin{formula}{#1}}
	\def\Ff{\end{formula}\noindent}

\def\BS{\color{Bittersweet}}

\newenvironment{Step}[2]{\smallskip
	{\em Step #1. #2}
	\newline \noindent}
	
\setstcolor{red}

\begin{document}
\title{\bf Stability of propagating terraces\\ in spatially periodic multistable equations in $\R^N$}
\author{Thomas Giletti\thanks{Univ.~Clermont Auvergne, LMBP UMR6620, Aubière, France} \ \
 \& \ Luca Rossi\footnote{Istituto G.~Castelnuovo, Sapienza Universit\`a di Roma, Rome, Italy}}
 \date{}
 
\maketitle



\abstract{
In this paper, we study the large time behaviour of solutions of multistable reaction-diffusion equations in~$\R^N$, with a spatially periodic heterogeneity. By multistable, we mean that the problem admits a finite -~but arbitrarily large~- number of stable, periodic steady states. In contrast with the more classical monostable and bistable frameworks, which exhibit the emergence of a single travelling front in the long run,
in the present case the large time dynamics is governed by a family of stacked travelling fronts, involving intermediate steady states, called propagating terrace.
Their existence in the multidimensional case has been established in our previous work~\cite{GR_2020}. 
The first result of the present paper is their uniqueness.
Next, we show that the speeds of the propagating terraces in different directions dictate the
spreading speeds of solutions of the Cauchy problem, for 
both planar-like and compactly supported initial data. The latter case turns out to be much more intricate than the former, due to the fact that the propagating terraces in distinct 
directions may involve different sets of intermediate steady states. 
Another source of difficulty is that the {\em Wulff shape} of the speeds of travelling fronts can be non-smooth,
as we show in the bistable case using a result of \cite{DingGiletti}.
}

\bigskip\bigskip

\noindent \emph{Keywords.} Reaction-diffusion equation; 
Multistable and spatially periodic equation; Propagating terrace;
Spreading speed; Wulff-shape

\medskip

\noindent\emph{MSC 2020.} Primary: 35K57; Secondary: 35C07, 35B40.

\normalsize

\bigskip

\section{Introduction}\label{sec:intro}

This paper is concerned with the reaction-diffusion equation
\Fi{eq:parabolic}
\partial_t u = \dv (A(x) \nabla u ) + f (x,u), \quad t >0, \ x \in \R^N.
\Ff
The diffusion matrix $A = (A_{i,j})_{1 \leq i,j \leq N}$ is assumed to be 
smooth and to satisfy
$$
\exists C_1 , C_2>0 , \quad \forall x , \xi \in \R^N, \quad C_1 |\xi|^2 \leq \sum_{i,j} A_{i,j} (x) \xi_i \xi_j \leq C_2 |\xi|^2  .
$$
We further assume $A$ and $f$ to be spatially periodic, with period $1$ in each direction of the canonical basis, namely,
$$\forall h \in \mathbb{Z}^N, \quad A (\cdot + h) \equiv A (\cdot) , \quad f (\cdot + h , \cdot) \equiv f(\cdot, \cdot).$$
From now on, when we say that a function is (spatially) periodic, we mean that it is invariant 
under translation by vectors in $\mathbb{Z}^N$.

We consider the case where equation~\eqref{eq:parabolic} is of the multistable type, in the following sense.
\begin{assump}\label{ass:multi}
The function $f : \R^N \times \R \mapsto \R$ is of class $C^1$ and equation~\eqref{eq:parabolic} satisfies the following:
\begin{itemize}
	\item there are two linearly stable, periodic steady states, that we call ``extremal'': 
	the trivial one $0$ and a positive one $\overline{p}$;
	\item any other periodic steady state between~$0$ and $\overline{p}$ is either linearly stable or linearly unstable.
\end{itemize}
\end{assump}

We emphasise that the linearly stable steady states are not assumed to be ordered,
namely, they are allowed to intersect each other. Of course, when two steady states are ordered, then 
they are strictly ordered, due to the elliptic strong maximum principle.
The precise meaning of linear stability and instability will be reclaimed in Section~\ref{sec:prelim}.
Assumption \ref{ass:multi} implies that the (linearly) stable, periodic steady states between~$0$ and $\overline{p}$ are isolated
and that their number is finite. This is shown in Proposition \ref{pro:finitep}
in Section~\ref{sec:prelim}. Linear stability implies {\em asymptotic stability} (cf.~Lemma~\ref{lem:perturb0} below) which,
for a steady state $p$, means the existence of 
an $L^\infty$-neighbourhood of $p$ such that any solution $u(t,x)$ of~\eqref{eq:parabolic} 
with an initial datum in such neighbourhood converges uniformly to $p$ as $t\to +\infty$. 
The largest of such neighbourhoods is called the {\em basin of attraction} of~$p$.

When $0$ and $\overline{p}$ are the unique stable steady states, 
 equation~\eqref{eq:parabolic} is said to be bistable. In such a case, 
 there typically exist particular entire in time solutions (i.e. solutions for all times), called {\em pulsating travelling fronts}.
 These solutions also characterise 
 the large time dynamics of solutions of the Cauchy problem. We refer to~\cite{SKT,Xin91-4} for early introductions of the concept in the literature, and to~\cite{DHZ17,Ducrot2,FZ} for recent developments in the bistable case. Here however, due to the possible existence of several stable steady states, 
  the propagation phenomenon may involve a stack of 
  fronts, which leads one to consider the more general notion of a {\em propagating terrace}. 
  The fact that propagating terraces arise in the large time behaviour of solutions of multistable equations
  has been proved in the case of spatial dimension equal to one~\cite{Terrace,MR4130256}, as well
  as in the multidimensional homogeneous setting~\cite{MR3705791,MR4078110,RossiFG}.
  Our goal is to extend those results to arbitrary dimension in the spatially periodic case.
  
  Let us recall the notions of pulsating travelling front (or wave) and propagating terrace.
\begin{defi}[Pulsating travelling front]\label{def:puls}
Let $q_1 > q_2$ be two periodic steady states of~\eqref{eq:parabolic} and let~$e \in S^{N-1}$. A pulsating travelling front connecting $q_1$ to $q_2$ in the direction~$e$ with speed $c \in \R$ is an entire solution to \eqref{eq:parabolic} that can be written in the form
$$u (t,x) = U (x, x\cdot e - ct),$$
where $U(x,z)$, called the profile of the front, is periodic in the $x$-variable and satisfies pointwise
$$U (\cdot, -\infty) \equiv q_1 (\cdot ) > U (\cdot, \cdot ) > U (\cdot, +\infty) \equiv q_2 (\cdot) .$$
\end{defi}
We point out that the particular case when $c=0$ must be handled carefully. Indeed, the change of variables $(t,x) \mapsto (x, x \cdot e- ct)$ is no longer invertible when $c=0$ and therefore the profile~$U$ cannot be inferred from~$u$, and it is actually relevant only on the set $\{(x,x\.e)\ |\ x\in\R^N\}$. In particular, parabolic estimates  imply that the convergence of~$U$ toward steady states is uniform with respect to the first variable only in the case $c \neq 0$; when $c =0$, one may instead deduce that $U(x,x \cdot e)$ converges to~$q_1$ and~$q_2$ respectively as $x \cdot e$ goes to~$-\infty$ and~$+\infty$. We refer to Proposition~\ref{prop:cv} below for more details. Furthermore, pulsating travelling fronts with zero speed enjoy weaker uniqueness and stability properties. As a matter of fact we will rule out this case in the present work.
\begin{defi}[Propagating terrace]\label{def:terrace}
A propagating terrace connecting $\overline{p}$ to~$0$ in a direction $e \in S^{N-1}$ is a couple of two finite sequences 
$$\mathcal{T} = ((q_k)_{0 \leq k \leq K} , (U_k)_{1 \leq k \leq K} ),$$
such that:
\begin{itemize}
\item the functions $q_k$ are periodic steady states of \eqref{eq:parabolic} with
$$ \overline{p} \equiv q_0 > q_1 > \cdots >q_K \equiv 0;$$
\item the functions $U_k$ are profiles of pulsating travelling fronts connecting $q_{k-1}$ to $q_k$ in the direction~$e$ with speed~$c_k$;
\item the sequence of speeds $(c_k)_{1 \leq k \leq K}$ is nondecreasing, i.e.
$$c_1 \leq c_2 \leq \cdots \leq c_K.$$
\end{itemize}
\end{defi}
We will occasionally refer to the steady states $(q_k)_{ 1 \leq k \leq K}$ of $\mc{T}$
as the platforms of the terrace, and to the integer $K$ as the number of platforms (notice that we do not include~$q_0$ in the platforms).
We will also sometimes say that a terrace contains a \ptf\ if it contains the associated profile $U$.

The notion of a propagating terrace appeared in~\cite{FMcL}, under the name of {\em minimal decomposition}, in the homogeneous case. More recently, it was studied in the one-dimensional spatially periodic case in~\cite{Terrace,MR4130256}, and in higher dimensions in our previous work~\cite{GR_2020}, whose main result we now recall.
%
\begin{theo}[\cite{GR_2020}]\label{th:existence}
Under Assumption~\ref{ass:multi}, for any $e \in \Sph$, there exists a propagating terrace $\mathcal{T} = ((q_k),(U_k))$ connecting $\overline{p}$ to $0$ 
in the direction $e$.

Furthermore, each platform $q_k$ of this terrace is a stable, periodic, steady state, and each profile $U_k(x,z)$ is nonincreasing with respect to~$z$.
\end{theo}
\begin{rmk}
Let us mention that Theorem~\ref{th:existence} was proven under slightly different hypotheses. 
First, it was explicitly assumed that there is a finite number of stable periodic steady states; here it instead follows from our Proposition~\ref{pro:finitep}. Furthermore, Assumption~1.3 in~\cite{GR_2020} involved a so-called ``counter-propagation''. That assumption states that, for any unstable periodic steady state $q$, the speed of any front connecting~$q$ to some lower periodic steady state is strictly less than the speed of any front connecting some larger periodic steady state to~$q$. However, if the non-stable steady states are linearly unstable, then
 by either~\cite[Theorem 2.4]{W02} or~\cite{BHR1}, the latter must be positive and the former negative, hence that inequality holds. 
Here we assume the stronger 
linear instability hypothesis in Assumption~\ref{ass:multi} for simplicity of the presentation. We refer to Section~\ref{sec:prelim} below for details.

Next, we point out that, due to the heterogeneity of the coefficients, which makes the problem not invariant by rotation,
the propagating terraces in two distinct directions are in general different. As we have shown in~\cite[Proposition 1.6]{GR_2020}, the intermediate steady states, and even the number of fronts~$K$, may also depend on the direction in a nontrivial way; see also~\cite{DingGiletti} for some other results on the asymmetry of fronts in spatially periodic reaction-diffusion equations.
\end{rmk}


\section{Statement of the main results}\label{sec:results}

This paper addresses the questions of uniqueness and attractiveness of the propagating terraces. 
Our results will confirm that, while there can be many families of stacked pulsating fronts, 
there exists only one propagating terrace in any given direction.
In addition, these
 terraces are the only ones appearing
in  the large time behaviour of solutions of the Cauchy problem, for two classes of initial data:
planar-like and compactly supported. 
This extends known results from the one dimensional
and the homogeneous~settings. 


\subsection{Uniqueness of the propagating terrace}\label{sec:!}

We start with the uniqueness result.
\begin{theo}\label{th:uniqueness}
Under Assumption~\ref{ass:multi}, let
$$\mathcal{T} = ((q_k)_{0 \leq k \leq K} , (U_k)_{1 \leq k \leq K})$$ be a terrace connecting $\overline{p}$ to $0$ in 
a direction $e \in \Sph$. If all its speeds are non-zero, i.e
\Fi{ck<>0}
c_k\neq0 \quad \text{for all }\; k\in\{1,\dots,K\},
\Ff
then $\mathcal{T}$ is the unique terrace in the direction~$e$, up to shifts, in the sense that for any other terrace $\mathcal{T}' = ((q'_k)_{0 \leq k \leq K'}, (U'_k)_{1 \leq k \leq K'})$ connecting $\overline{p}$ to $0$ in the direction $e$, it holds that $K' = K$ and 
$$\forall k\in\{0,\dots,K\}, \quad q_k \equiv q'_k ,$$
$$\forall k\in\{1,\dots,K\}, \quad \exists \xi_k \in \R, \quad U_k ( \cdot , \cdot + \xi_k) \equiv U'_k (\cdot, \cdot).$$
\end{theo}

Theorem~\ref{th:uniqueness} expresses the rather surprising fact that there is only one way of covering the range $0$, $\ol p$
by a family of ordered fronts with ordered (non-zero) speeds.
In Section~\ref{sec:algorithm} we will exhibit an iterative procedure for the determination of the fronts ``selected'' by the terrace.

We point out that the assumption that the speeds $c_k \neq 0$ is truly necessary. When some of the fronts have zero speed, there may exist different terraces and even possibly with a different number of platforms; we refer to~\cite[Section~6]{MR4130256} for details in the one dimensional case.


\subsection{Planar-like initial data}\label{sec:planar-like}

Next, we investigate the attractiveness of the (unique) propagating terrace for the Cauchy problem when the initial datum is ``planar-like''.
Throughout the the paper, initial data are always assumed to be measurable functions between $0$ and $\ol p$. 
We show that if the initial datum is ``planar-like'' in a given direction, 
then the unique propagating terrace in that direction emerges 
(through its speeds and steady states) in the large time behaviour of the solution.
\begin{theo}\label{th:planar_speeds}
Under Assumption~\ref{ass:multi}, let $\mathcal{T} = ((q_k)_{0 \leq k \leq K} , (U_k)_{1 \leq k \leq K})$
be a terrace connecting $\overline{p}$ to $0$ in 
a direction $e \in \Sph$. Assume that all its speeds are non-zero.
%
Let~$u$ be a solution with an initial datum $0 \leq u_0 \leq \overline{p}$ satisfying
$$\liminf_{x \cdot e \to -\infty} \big(u_0 (x)- \overline{p} (x)\big)>-\eta,\qquad
\limsup_{x \cdot e \to +\infty} u_0 (x)<\eta,$$
where $\eta>0$ is such that $\overline{p} - \eta$ and $\eta$
lie in the basins of attraction of $\overline{p}$ and~$0$ respectively. 

Then~$u$ spreads in the direction $e$
accordingly to the terrace $\mathcal{T}$, 
in the sense that, for any $\eps>0$ and $1 \leq k \leq K$, it holds
$$ 
\liminf_{t \to +\infty} 
\bigg(\inf_{ x \cdot e \leq (c_k - \eps ) t} \big(u(t,x) - q_{k-1} (x)\big)\bigg)  \geq 0,$$
$$ 
\limsup_{t \to +\infty} 
\bigg(\sup_{ x\cdot e\geq (c_k + \eps) t} \big( u(t,x) - q_k (x)\big)\bigg) \leq  0.$$
In particular, if $c_k < c_{k+1}$ then one has
%
%
%
%
\begin{equation}\label{eq:plateau}
\lim_{t \to +\infty} 
\bigg(\sup_{(c_k + \eps) t \leq x \cdot e \leq (c_{k+1} - \eps ) t} \big| u(t,x) - q_k (x) \big|\bigg) = 0 .
\end{equation}
\end{theo}
\begin{rmk}\label{rmk:theo_planar}
Several comments are in order. 
\begin{enumerate}
\item It is reasonable to expect 
that the conclusions of Theorem~\ref{th:planar_speeds} hold true when some of the speeds~$c_k$ are zero. However, our construction of sub and supersolutions does not cover this situation, the reason being related to the non-invertibility of the change of variables in Definition~\ref{def:puls}. 
A different argument is required, possibly by comparison with some other multistable reaction-diffusion equations whose terraces only involve non-zero speeds.
\item One may wonder whether the solution also converges, 
in the moving frames with speeds $c_k$, to the profiles $U_k$ of the travelling fronts. 
Actually, under the additional assumption that the speeds are not only non-zero but also strictly ordered, 
i.e. $0 < c_k < c_{k+1}$ for all $k\in\{1,\dots,K-1\}$, one could refine
the arguments of our proof by constructing some sharper sub and supersolutions,
in the spirit of Fife and McLeod~\cite{FMcL}. This would show that the solution is 
asymptotically ``trapped'' between two shifts of the propagating terrace, 
in the sense that, for any $k\in\{1,\dots,K\}$, 
$$\limsup_{t \to +\infty} \sup_{ (c_k-\eps) t  \leq x \cdot e  \leq (c_k+\eps) t } 
\big(u (t, x)  - U_k (x, x\cdot e- c_k t - \xi  )\big) \leq 0,$$
$$\liminf_{t \to +\infty} \inf_{ (c_k-\eps) t  \leq x \cdot e  \leq (c_k+\eps) t }  
\big(u (t, x)  - U_k (x, x\cdot e- c_k t + \xi )\big) \geq 0 ,$$
for any given $0<\eps<\min(c_k-c_{k-1},c_{k+1}-c_k)$ and $\xi>0$ sufficiently large. Since the speeds~$c_k$ are non-zero, one may then invoke a result of Berestycki and Hamel on generalised transition fronts~\cite{BH12}. One would deduce that the solution converges
as $t\to+\infty$, in the moving frame with speed~$c_k$ in the direction~$e$ and up to extraction of a time subsequence, to some shift of the pulsating front profile~$U_k$. Yet we expect that this convergence may not be uniform in space, or may not even occur for any arbitrary sequence of times. Since this 
 improvement would give raise to significant additional technicalities, such as an estimation of the exponential decay of the fronts, we leave it to further investigation.
\item 
The emergence of a zone where the solution converges to one of the steady states~$q_k$
is only guaranteed by~\eqref{eq:plateau} in Theorem~\ref{th:planar_speeds}
when $c_k < c_{k+1}$, with ${1\leq k\leq K-1}$.
Notice that convergence towards the extremal steady states $0$ and $\overline{p}$ is already contained in the first part of 
Theorem~\ref{th:planar_speeds},
which indeed implies
$$ \forall c<c_1,\qquad
\lim_{t \to +\infty} 
\bigg(\sup_{ x\cdot e\leq c t} |u(t,x)-\overline{p}(x)|\bigg)=  0,$$
$$ \forall c>c_K,\qquad
\lim_{t \to +\infty} 
\bigg(\sup_{ x\cdot e\geq c t} | u(t,x) | \bigg)=  0.$$
because $0\leq u\leq \overline{p}$ by the maximum principle.
In the remaining case $c_k = c_{k+1}$, owing to the heuristics that the pulsating travelling fronts of the terrace should 
appear in the large time behaviour of the solution, none of which might connect directly~$q_{k-1}$ and $q_{k+1}$, 
thanks to the uniqueness of the propagating terrace given by Theorem~\ref{th:uniqueness},
we believe that such a zone should still emerge in the long run, but only expanding sublinearly in time.
\end{enumerate} 
\end{rmk}


\subsection{Compactly supported initial data}\label{sec:cptsupp}

We then turn to the case where the initial datum $u_0$ has a compact support. It is expected that {\em invasion} occurs, 
i.e.~that the solution converges locally uniformly to the largest stable steady state~$\overline{p}$, 
provided that all speeds are positive and the initial datum is large enough, in a suitable sense. So there will be an expanding 
region where the solution converges towards~$\overline{p}$.
But there could also exist some regions where the solution is attracted by other stable states. 
Our goal is to describe the asymptotic shape of all these regions as $t\to+\infty$.
This is achieved in \cite{DMpreprint,RossiFG} in the case of the autonomous equation 
$\partial_t u = \Delta u+ f (u)$.
The case of the general periodic equation \eqref{eq:parabolic} is much more complex, also compared with
that of planar-like initial data considered in the previous subsection.
The complexity comes from two factors, which are peculiar to heterogeneous equations in dimension higher than one:
	\begin{enumerate}[\bf 1]
	\item The {\em spreading shape} associated with compactly supported initial data is related to the spreading speeds of planar-like solutions in different directions 
	in a nontrivial way, and precisely through their {\em Wulff shape}.
	
	\item Propagating terraces in distinct directions
may themselves have very different forms;
we have indeed constructed in~\cite[Proposition 1.6]{GR_2020} an explicit example in which the number of 
platforms of the terrace (i.e.~the number $K$ in Definition~\ref{def:terrace}) changes as the direction varies.
\end{enumerate}

As a matter of fact, factor {\bf 1} already arises for the bistable equation,
that is when $0$ and $\overline{p}$ are the only linearly stable steady states, and any intermediate steady state is linearly unstable. 
According to Theorem~\ref{th:existence}, for any direction $e\in S^{N-1}$, there exists a 
travelling wave connecting~$\overline{p}$ to~$0$, with some speed $c^* (e)$.
Assuming that $c^*(e)>0$ for all $e\in\Sph$, one has that
the {\em spreading speed} in any given direction $e\in S^{N-1}$ exists and is given by the
Freidlin-G\"artner formula
\Fi{w*}
w^* (e) = \inf_{\substack{ e' \in S^{N-1}\\ e' \cdot e >0}} \frac{c^* (e')}{e' \cdot e},
\Ff
or, equivalently, that the so-called {\em spreading shape} is given by the {\em Wulff shape}
of the speeds of the fronts:
\[\begin{split}
W^* &= \bigcap_{e \in S^{N-1}} \{ x\in\R^N\ |\ x \cdot e \leq c^* (e) \} \\
&= \{ r e \ | \ e\in S^{N-1} \text{\; and \;}  0 \leq r \leq w^* (e) \},
\end{split}\]
with $w^*(e)$ given by~\eqref{w*}.
More precisely, \cite[Theorems 1.4 and 1.5]{RossiFG}
assert that, for a solution $u$ with a compactly supported initial datum, which satisfies
$u(t,x)\to\ol p$ as $t\to+\infty$ locally uniformly in $x$,
the following hold for any $\eps\in(0,1)$:
$$
\limsup_{t \to +\infty} \bigg(
 \sup_{x    \in\, \R^N\setminus   (1+\varepsilon)t  W^* } u(t,x) \bigg)= 0,
\qquad
\limsup_{t \to +\infty}  \bigg( \sup_{x   \in  (1-\varepsilon)t  W^*  } | u(t,x) -   \overline{p} (x) |  \bigg) = 0.
$$
In other words, the region where the solution approaches the upper steady state~$\overline{p}$ is the intersection (among all directions) of the half-spaces where travelling fronts are close to~$\overline{p}$. We also refer to~\cite{FG} for the origin of this formula under the Fisher-KPP assumption, i.e. $0 < f(x,u) \leq \partial_u f (x,0) u$ for $0 < u < \overline{p}$. In the homogenous case, where $c^* (e) = c^*$ is independent of~$e$, the Freidlin-G\"artner 
gives $w^* (e) = c^* $ and the spreading shape~$W^*$ reduces to the ball $B_{c^*}$; thus the formula is consistent with and extends the classical results of the homogeneous scalar equation~\cite{AW}.

As for factor {\bf 2}, it is peculiar to the multistable case.
As already seen in the planar case, the speed of propagation may vary from one level set to another,
according to the propagating terrace. 
Therefore, one needs to deal with  {\em several spreading shapes}, associated with different levels.
This becomes very intricate when the platforms involved in the terraces, and even their number, 
depend on the direction of the terrace. The asymptotic spreading shape is then a result of a complex interplay between 
distinct
platforms, some of which might appear in some of the terraces but not in others.

We will present below two results on the large time spreading of solutions with compactly supported initial data, under two different set of assumptions. As a matter of fact, our first set of assumptions will rule out factor {\bf 2} by the hypothesis that all the terraces, in all directions, share the same platforms. Our second set of assumptions will instead be related with the regularity of the spreading shapes associated with terraces. In both cases, we derive some generalisations of the
Freidlin-G\"{a}rtner formula. They show that 
solutions with compactly supported initial data spread 
towards distinct steady states,
exhibiting a family of nested asymptotic shapes. Each asymptotic shape is given by 
the Wulff shape associated with the speeds of the terraces corresponding to those steady states. This leads us to introduce the more general notion below.
%
\begin{defi}[Wulff shape]\label{def:W}
The Wulff shape associated with a function $c:\Sph\to\R$ is
	$$W_c:=\bigcap_{e \in S^{N-1}} \{ x\in\R^N\ |\ x \cdot e \leq c(e) \}.$$
\end{defi}

Let us state our first set of assumptions for the case when the initial datum is compactly supported.
\begin{assump}\label{ass:same_shape}
	Equation \eqref{eq:parabolic} is of the multistable type, in the sense of Assumption~\ref{ass:multi}, and, for any~$e\in S^{N-1}$,
	there is a terrace
	$$\mathcal{T}^e = ((q_k^e )_{0\leq k \leq K^e}, (U_k^e)_{1 \leq k \leq K^e} )$$
	connecting $\overline{p}$ to $0$ in the direction $e$,
	which satisfies:
	\begin{itemize}
		\item $c_1(e) >0$\, for all $e \in S^{N-1}$, where $c_1(e)$ is the speed of the front with profile~$U_1^e$;
		\item the integer $K^e = K$ and the platforms $q_k^e = q_k$, $1 \leq k \leq K$, do not depend on $e \in S^{N-1}$.
	\end{itemize}
\end{assump}
\noindent Several comments about Assumption~\ref{ass:same_shape} are in order. 
First, Assumption~\ref{ass:multi} guarantees the existence of the terraces in all directions, owing to Theorem~\ref{th:existence}.
Second, the first condition implies that all travelling fronts of the terrace~$\mathcal{T}^e$ have positive speeds, for any direction $e \in S^{N-1}$, hence Theorem~\ref{th:uniqueness} entails that the terrace $\mathcal{T}^e$ is unique, up to shifts of the fronts. 
Third, we emphasise that the set of steady states $(q_k)_{0 \leq k \leq K}$ does not necessarily coincide with the set of all stable steady states of the~equation.

We are now in a position to state our first result on the spreading of solutions associated with compactly supported initial~data.
\begin{theo}\label{th:spread_cpct1}
	Under Assumption~\ref{ass:same_shape}, let $u$ be a solution emerging from a compactly supported initial datum $0 \leq u_0 \leq \overline{p}$,
	for which there is invasion, i.e.
	\begin{equation}\label{spread_start}
		\lim_{ t \to +\infty } u (t, x) \rightarrow \overline{p} \quad\text{ locally uniformly in }\;x\in\R^N.
	\end{equation}
	Then the spreading shapes of $u$ are given by the Wulff shapes $W_{c_k}$ of the speeds $c_k$ 
	of the terraces~$(\mathcal{T}^e)_{e \in S^{N-1}}$, in the following sense: for every 
	$k\in\{1,\dots,K\}$ and $\eps \in (0,1)$, it holds
		$$
		\limsup_{t \to +\infty}
		\bigg( \sup_{x   \in\, \R^N\setminus   (1+\varepsilon)t  W_{c_k}  } \big(u(t,x) -  q _k (x)\big)\bigg) \leq 0,$$
		$$
		\liminf_{t \to +\infty}\bigg( \inf_{x   \in  (1-\varepsilon)t  W_{c_k}  } \big(u(t,x) -  q_{k-1}(x)
		\big)\bigg) \geq 0.
		$$
	In particular, for $k=0,K$ and all $k\in\{1,\dots,K-1\}$ such that $c_k\neq c_{k+1}$, and any $\eps \in (0,1)$, one has
	$$\lim_{t \to +\infty} 
	\bigg(\sup_{x \in (1-\varepsilon) t W_{c_{k+1}}  \setminus (1+\varepsilon)t  W_{c_k}  } | u(t,x) -  q _k (x) |\bigg) = 0,$$
	where, in the cases $k=0,K$, it is understood that $W_{c_0}:= \emptyset$ and $W_{c_{K+1}} := \R^N$.
\end{theo}

We point out that 
the invasion condition~\eqref{spread_start} cannot be fulfilled by all solutions.
		For instance, if the initial datum $u_0$ lies in the basin of attraction of~$0$, then the solution would simply 
		converge to $0$ uniformly as $t \to +\infty$.
		On the other hand, we will show that the validity of~\eqref{spread_start} for ``large enough'' initial data is equivalent
		to the existence of a terrace, in any given direction, whose uppermost travelling front has
		positive speed (the terrace is then unique --~up to translation of the fronts~--
		thanks to Theorem~\ref{th:uniqueness}), see Lemma~\ref{lem:spread_ini} and 
		Proposition~\ref{pro:spread_ini} below.
		``Large enough'' here means being close enough to~$\overline{p}$ on a sufficiently large ball. 
Then, in particular, Assumption~\ref{ass:same_shape} guarantees the invasion property~\eqref{spread_start} 
for large enough initial data, whence requiring it boils down to a condition on the initial datum, rather than an
additional hypothesis on the equation.

Unfortunately, Assumption~\ref{ass:same_shape} is not always fulfilled. Indeed, in~\cite[Proposition 1.6]{GR_2020}, we have exhibited an example of an equation for which 
the terrace in a direction is composed by two fronts, and the one in another 
direction is composed by only one front. Hence Theorem~\ref{th:spread_cpct1} does not fully solve the question of 
spreading speeds for compactly supported initial data in the multistable spatially periodic framework.

Under only Assumption~\ref{ass:multi}, we are able to  derive a lower and an upper estimate on the spreading shapes,
but not to prove that those estimates always fit together to provide the precise description of the spreading shapes.
Still, there is also another situation, besides Assumption~\ref{ass:same_shape}, in which we are able to
derive a complete characterisation of the spreading shapes. The result naturally involves the Wulff shapes of the speeds of fronts connecting intermediate states.
	Actually, we will also need to consider terraces whose upper state is not $\ol p$.
	More precisely, we introduce the following.

\begin{defi}\label{def:Wp}
		Let $0 \leq p \leq \ol p$ be a linearly stable, periodic steady state.
		\begin{enumerate}[$(i)$]
				\item 
		If $p\not\equiv \ol p$, then we denote by $c[p] (e)$ the speed of the travelling front belonging to 
		a terrace $\mathcal{T}^e$ connecting $\overline{p}$ to~$0$ in a 
		direction~$e\in\Sph$,
		whose profile fulfils
		\begin{equation}\label{eq:through}
			U (\cdot, +\infty) \leq p (\cdot),\qquad
			\max_{x\in\R^N}(U (x, -\infty)- p(x)\big)>0.
		\end{equation}
		
		\item If $p\not\equiv 0$, then we denote by	$c_1^{p}(e)$ the speed of the uppermost front of a terrace $\mathcal{T}^e_{p}$ 
		connecting~$p$ to~$0$ in the direction $e$.
	\end{enumerate}
\end{defi}

The existence of a terrace $\mathcal{T}_p^e$ as in Definition~\ref{def:Wp} 
is still given by Theorem \ref{th:existence},
because the multistable Assumption~\ref{ass:multi} is preserved.
The platforms of $\mathcal{T}^e_p$ and $\mathcal{T}^e$ may be different.
We point out that the definitions of the speeds $c[p] (e)$ and~$c_1^{p}(e)$ depend on the choice of the terraces $\mathcal{T}^e$ and~$\mathcal{T}^e_{p}$.
However, thanks to our uniqueness result Theorem~\ref{th:uniqueness}, 
these speeds are well defined when respectively
$c_1^{\overline{p}}(e)$ and~$c_1^{p}(e)$ are positive (for the a-fortiori unique corresponding terrace),
and this will be the case in our application below.
Observe that if $p$ is the $k$-th platform of the terrace $\mathcal{T}^e$, then $c[p] (e)$ is the speed of the $k$-th front of $\mathcal{T}^e$. In particular, our last main result will be consistent with Theorem~\ref{th:spread_cpct1}.

The key assumption we need is that the Wulff shapes of the speeds~$c_1^{p}$ are of class~$C^1$. 

\begin{assump}\label{ass:tangentialC1}
	Equation \eqref{eq:parabolic} is of the multistable type, in the sense of Assumption~\ref{ass:multi}, 
	and moreover all (linearly) stable, periodic steady states between~$0$ and $\overline{p}$ are totally ordered, 
	i.e. they are given~by
	$$p_0 \equiv \overline{p} > p_1 > \cdots > p_M \equiv 0,$$
	for some positive integer $M$. 
%
Furthermore, the speeds $c_1^{p_k}$ given by Definition~\ref{def:Wp} satisfy:
	\begin{itemize}
		\item $c_1^{\overline{p}}(e) >0$ for all $e \in S^{N-1}$;
		\item for any $k\in\{0,\dots,M-1\}$ such that the function $c_1^{p_k}$ is strictly positive,  
		the Wulff shape $W_{c_1^{p_k}}$ is of class $C^{1}$.
	\end{itemize}
\end{assump}
\noindent 
We remark that the fact that the number of
stable, periodic steady states between~$0$ and $\overline{p}$ is finite is already granted by Assumption~\ref{ass:multi}, see
Proposition \ref{pro:finitep} below. Assumption~\ref{ass:tangentialC1} additionally requires that they are ordered.
Similarly to Assumption~\ref{ass:same_shape}, the positivity of the speeds $c_1^{\overline{p}}(e)$ 
ensures that all travelling fronts of the terrace $\mathcal{T}^e = \mathcal{T}_{\overline{p}}^e$ connecting $\overline{p}$ to~$0$ in any direction~$e$ have positive speeds, and in particular this terrace is unique up to shifts. The last condition in 
Assumption~\ref{ass:tangentialC1} is also well posed, because it only involves the speeds $c_1^{p_k}$ that are positive,
hence again they are uniquely defined.
Such condition is more intricate as it involves the speeds of travelling fronts which do not necessarily belong to the terrace connecting~$\overline{p}$ to~$0$. We believe that it is
in some sense ``generic'', since Wulff shapes appear to be smooth 
in numerical simulations. Yet, it is actually possible to construct counter-examples, as we will see in Corollary \ref{ce:C1}.

We also point out that, while the speeds $c_1^{p_k}$ are involved in Assumption~\ref{ass:tangentialC1}, it ultimately turns out that they 
do not appear in the spreading shapes of solutions, provided by the next result.

\begin{theo}\label{th:spread_cpct11}
	Under Assumption~\ref{ass:tangentialC1}, 
	let $u$ be a solution emerging from a  compactly supported 
	initial datum $0 \leq u_0 \leq \overline{p}$
	for which~\eqref{spread_start} holds.
	
	Then the spreading shapes of $u$ are given by the Wulff shapes of the speeds 
	of the terraces~$(\mathcal{T}^e)_{e \in S^{N-1}}$, in the following sense: for every $k\in\{1,\dots,M\}$ and $\eps\in(0,1)$, it holds
	$$\limsup_{t \to +\infty}
	\bigg( \sup_{x   \in\, \R^N\setminus   (1+\varepsilon)t  W_{c [p_k]}  } \big(u(t,x) -  p_k (x)\big)\bigg) \leq 0,$$
	$$
	\liminf_{t \to +\infty}\bigg( \inf_{x   \in  (1-\varepsilon)t  W_{c [p_k]}  } \big(u(t,x) -  p_{k-1} (x)
	\big)\bigg) \geq 0,
	$$
	where the $c[p_k]$ are given by Definition~\ref{def:Wp}. 
	In particular, for all $k\in\{0,\dots,M\}$ and any $\eps \in (0,1)$, one has

	$$
	\bigg(\sup_{x \in (1-\varepsilon) t W_{c [p_{k+1}]}  \setminus (1+\varepsilon)t  W_{c [p_{k}]}  } | u(t,x) -  p_k (x) |\bigg) = 0,$$
	where, in the cases $k=0,M$, it is understood that $W_{c [p_0]}:= \emptyset$ and $W_{c [p_{M+1}]} := \R^N$.
\end{theo}
Notice that a given steady state~$p_k$ in Theorem~\ref{th:spread_cpct11} may not necessarily be a platform of the propagating terraces, and whether this is the case or not may also depend on the direction. In particular, it may happen that $W_{c[p_k]} \equiv W_{c[p_{k+1}]}$ for some~$k$, 
in which case the convergence towards~$p_k$ given by the last statement of Theorem~\ref{th:spread_cpct11} 
would occur on an empty set.

\begin{rmk}\label{rmk:rmk} Let us again make several comments.
	\begin{enumerate}[$(i)$]
		\item As discussed before, according to Lemma~\ref{lem:spread_ini} and 
			Proposition~\ref{pro:spread_ini} below, the validity of the invasion condition~\eqref{spread_start}
		for ``large enough'' initial data is equivalent to the positivity of the speeds~$c_1^{\overline{p}}(e)$,
		which is granted here by Assumption~\ref{ass:tangentialC1}.
\item To recover a Freidlin-G\"artner formula from e.g.~Theorem \ref{th:spread_cpct11}, one observes 
that the sets~$W_{c [p_k]}$, for $1 \leq k \leq M$, 
rewrite~as
$$W_{c [p_k]} = \{ r e  \  | \ e \in S^{N-1} \text{ and } 0 \leq r \leq w_k (e)  \},$$
with
$$
w_k (e) := \min_{\substack{ e' \in S^{N-1}\\ e' \cdot e >0}} \frac{c [p_k] (e')}{e'\cdot e} .
$$
In particular, Theorem~\ref{th:spread_cpct11} implies that, for any $e \in S^{N-1}$,
the following limits hold, locally uniformly in $x\in\R^N$:
\[\begin{array}{rl}
\displaystyle \forall  c\in[0,w_1 (e)) , & 
\displaystyle \lim_{ t \to +\infty}| u(t, x +ct e)-\overline{p}( x +cte )|=0,\vspace{3pt}\\
\displaystyle \forall 1 \leq k \leq M-1,  \forall  c \in (w_k (e) , w_{k+1} (e)) ,  & 
\displaystyle \lim_{ t \to +\infty} | u(t, x +ct e)-p_k( x +ct e)|=0 ,\vspace{3pt}\\
\displaystyle \forall  c > w_M (e) , & 
\displaystyle \lim_{ t \to +\infty} u(t, x +cte ) = 0.
\end{array}\]
Roughly speaking, $w_{k+1}(e)$ is the asymptotic speed of propagation in direction~$e$ of the set where
$u\sim p_k$ . In the case $M=1$, then $c [p_1] (e)$ is the speed of the travelling front connecting~$\overline{p}$ to~$0$, and we recover the 
result known in the bistable setting~\cite{RossiFG}. 
Therefore, our result extends the standard Freidlin-G\"artner formula to multistable equations.
\item We point out that the aforementioned difficulties {\bf 1} and {\bf 2}, that led to the additional hypotheses in either Assumptions~\ref{ass:same_shape} or~\ref{ass:tangentialC1}, are related to the 
dimension. For instance, in dimension $N=1$, then the minimum in the Freidlin-G\"{a}rtner formula~\eqref{w*} 
for the multistable case is trivial as it is taken over a singleton set. As a matter of fact, 
when~$N=1$,
Assumption~\ref{ass:multi} and the positivity of the speeds are enough to establish a leftward and rightward
spreading result 
for compactly supported initial~data. The proof basically proceeds as in the planar-like case, almost independently in each of the opposite directions. For the sake of conciseness, we chose not to include it~here.
		\item  In a situation when both Assumptions~\ref{ass:same_shape} and~\ref{ass:tangentialC1} hold true, then  Theorems~\ref{th:spread_cpct1} and~\ref{th:spread_cpct11} align with each other. Indeed, according to Theorem~\ref{th:existence}, any platform of a terrace is a stable periodic steady state. This means that the $(q_k)_{ 0 \leq k \leq K}$ from Assumption~\ref{ass:same_shape} must be a subset of the $(p_k)_{ 0 \leq k \leq M}$ from Assumption~\ref{ass:tangentialC1}. More precisely, there exist integers $0 = m_0 <  \cdots < m_K = M$ such that $q_k = p_{m_k}$ for all $0 \leq k \leq K$. Then $c_{[p_{m}]}= c_k $ for any $1 \leq k \leq K$ and $m_{k-1} < m \leq m_k$; 
			eventually the conclusions of Theorems~\ref{th:spread_cpct1} and~\ref{th:spread_cpct11} turn out to be identical.
\end{enumerate}
\end{rmk}

In conclusion, in the general multistable, spatially periodic case, we are only able to derive
a lower and an upper bound on the spreading speeds in all directions, cf.~Propositions~\ref{spread_upper} and~\ref{spread_lower_induc} below. We do not know whether these bounds always coincide, and if not, what would then be the exact spreading speeds.
However, let us stress that the situations where our theorems do not apply, that is, 
where neither Assumptions~\ref{ass:same_shape} nor~\ref{ass:tangentialC1} hold,
are when it simultaneously happens that the number of platforms of the terraces varies depending on the direction, and 
moreover some of the Wulff shapes are not regular.
We believe that these situations are rather pathological.

In such a pathological case, the difficulty seems to boil down to the following. Without Assumption~\ref{ass:tangentialC1}, it may occur that the expected spreading Wulff shape is not smooth, so that at some point it has two distinct tangential hyperplanes. This suggests that around this point, the solution may simultaneously approach two distinct travelling fronts in different directions, which without Assumption~\ref{ass:same_shape} may connect different steady states. We do not know how to handle such an intricate situation, nor can we rule out that it occurs. Therefore, to our knowledge, the issue of the spreading shape in the general multistable remains an open problem to be addressed in a future work.

\paragraph{Plan of the paper.}
In Section~\ref{sec:planar}, we prove both the uniqueness of the propagating terrace in any given direction, and our spreading result for planar-like initial data. The reason we gather both proofs in a single section is that they proceed quite similarly, by ``gluing'' some well chosen perturbations of the travelling front parts of the terrace into sub and supersolutions. We also sketch a procedure for the determination of the platforms of the terrace; 
although not used in our proofs, 
	this provides insight on why some steady states are ``selected'' by the terrace and others are not.

Sections~\ref{sec:cpct1} and~\ref{sec:cpct} deal with the case of compactly supported initial data, respectively under Assumptions~\ref{ass:same_shape} and~\ref{ass:tangentialC1}. We will first show some lower and upper bounds on the spreading speeds in the general case, which will quickly turn out to coincide under Assumption~\ref{ass:same_shape}. In Section~\ref{sec:cpct} we will work under Assumption~\ref{ass:tangentialC1}, where spreading speeds can again be obtained. Due to the technicality of these two set of assumptions, we will finally provide a short discussion in Section~\ref{sec:assumptions} and point out to some counter-examples to either.



\section{The planar case}\label{sec:planar}

In this section, we derive both the uniqueness of the propagating terrace, 
Theorem~\ref{th:uniqueness}, as well as
the asymptotic spreading speed of solutions associated with planar-like initial data, Theorem~\ref{th:planar_speeds}.
The two proofs rely on rather similar arguments.

\subsection{Preliminaries}\label{sec:prelim}

The linear stability (resp.~instability) of a periodic steady state~$p$ means 
that the periodic principal eigenvalue $\lambda^p$ of the linearised operator around~$p$
is strictly negative (resp.~positive),
the latter being defined as the unique real number
for which the following eigenvalue problem admits a solution:
\Fi{periodicpe}
\left\{ 
\begin{array}{l}
	\dv (A(x) \nabla \varphi ) + \partial_u f (x, p (x))  \varphi  = \lambda^p \varphi , \quad   x \in \R^N,\vspace{3pt}\\
	\varphi \text{ is periodic and positive}.
\end{array}
\right.
\Ff
The existence of $\lambda^p$, together with its simplicity, is provided by the Krein-Rutman theory~\cite{KR}.
The function $\vp$ is called the periodic principal eigenfunction, and we will always assume the normalisation condition
$$\max_{\R^N} \varphi = 1.$$

We start with the proof of the well known fact that the linear stability allows one to perturb 
solutions to get sub and supersolutions, in particular yielding the asymptotic stability.
\begin{lem}\label{lem:perturb0}
Let $p$ be a linearly stable, periodic steady state of~\eqref{eq:parabolic} and 
let $\vp$ be the 
	periodic principal eigenfunction of the problem~\eqref{periodicpe} 
	normalised by $\max_{\R^N} \varphi = 1$. 
	Then there exists $\delta >0$ such that, for any 
	 $\eta,\sigma\in[0,\delta]$,   
the following hold:
	\begin{enumerate}[$(i)$]
		\item any function $u_0$ satisfying $|u_0-p|\leq \delta\vp$ 
		lies in the basin of attraction of~$p$, that is, the solution emerging from it
		converges uniformly to~$p$ as~$t\to+\infty$;
		\item if $u$ is a supersolution (resp.~subsolution) of~\eqref{eq:parabolic}, then the function 
		$$u^\eta:=u+\eta \vp e^{-\sigma t}\qquad\text{(resp.~$u^\eta:=u-\eta \vp e^{-\sigma t}$)}$$
		satisfies
		$$
		\partial_t u^\eta \geq \dv (A(x) \nabla u^\eta ) + f (x,u^\eta) + \delta \eta e^{-\sigma t} $$
		$$\text{(resp. \,$\partial_t u^\eta \leq \dv (A(x) \nabla u^\eta ) + f (x,u^\eta) - \delta \eta e^{-\sigma t}$\,)}
		$$
		for all $t,x$ such that $|u(t,x)-p(x)|\leq\delta$.
	\end{enumerate}
	%
\end{lem}

\begin{proof}
	Statement $(i)$ immediately follows from $(ii)$, by taking
	$u^\eta:=p\pm\delta \vp e^{-\delta t}$ and applying the comparison principle.
	
	Let us prove statement $(ii)$ when $u$ is a supersolution. 
	The subsolution case then follows by considering  
	$\overline{p}(x) - u(t,x)$, which is a supersolution  
	of \eqref{eq:parabolic} with $f (x,u)$ replaced by
	$f (x, \overline{p} (x) - u)$ 
	(the linear stability of $p$ is inherited by $\ol p- p$).
	
	
	Since $f$ is of class $C^1$ and it is spatially periodic, 
	for $\delta' >0$ small enough it holds that
	\Fi{fC1}
	|s-p (x)|\leq 2 \delta' \ \implies \ 
	| \partial_u f (x, s) - \partial_u f (x, p(x)) | \leq- \frac{ \lambda^p}{2}.
	\Ff
	%
	Let~$u$ be a supersolution of~\eqref{eq:parabolic} satisfying 
	$|u (t,x)-p(x)|\leq\delta'$ for some $t>0$ and $x\in\R^N$.
	For $\eta,\sigma\geq0$,~set
	$$u^\eta:=u+\eta\vp e^{-\sigma t}.$$ We compute
	\begin{align*}
		\partial_t u^\eta - \dv (A(x) \nabla  u^\eta) &=
		f(x,u)-  \eta e^{-\sigma t}( \dv (A (x) \nabla \varphi) +\sigma\varphi)\\
		& = f(x,u)+\big[\partial_u f(x,p(x))-\lambda^p-\sigma)\big]\eta\varphi e^{-\sigma t}.
	\end{align*}
	Restricting to 
	$\eta\leq\delta'$ we have $|u^\eta(t,x)-p(x)|\leq2\delta'$
	and thus using~\eqref{fC1} and $\lambda_p<0$ we~derive 
	$$\partial_t u^\eta - \dv (A(x) \nabla  u^\eta) - f (x,  u^\eta)
	\geq  -\Big(\frac{\lambda^p}2+\sigma\Big)\eta\vp e^{-\sigma t}.$$
	Imposing $\sigma\leq-\lambda^p/4$ we get
	$$-\Big(\frac{\lambda^p}2+\sigma\Big)\eta\vp e^{-\sigma t}
	\geq-\frac{\lambda^p}4\,\eta\vp e^{-\sigma t}.$$
	Taking 
	$$\delta:=\min\Big\{\delta',-\frac{\lambda^p}4\min\varphi\Big\},$$
	this proves the result.
	%
\end{proof}

Lemma \ref{lem:perturb0} readily implies the following result, announced in the Introduction.

\begin{prop}\label{pro:finitep}
	Under Assumption \ref{ass:multi}, the number of stable,
	periodic steady states between~$0$ and~$\ol p$ is finite.
\end{prop}
\begin{proof}
	Consider a stable, periodic steady state $\t p$ between $0$ and $\ol p$.
	By Assumption \ref{ass:multi}, $\t p$ is linearly stable.
	Then by Lemma \ref{lem:perturb0},
	there is a $L^\infty$ neighbourhood of $\t p$ which belongs to the basin of attraction of $\t p$.
	Thus,  in this neighbourhood, there cannot be other steady states than $\t p$.
	
	Assume now by contradiction that there is an infinite number of stable
	periodic steady states between $0$ and $\ol p$.
	Then, by parabolic estimates, there exists a sequence of distinct stable,
	periodic steady states that converges uniformly 
	to some periodic, steady state $\t p$ between~$0$ and $\overline{p}$. 
	By the continuity of the periodic principal eigenvalue with respect to the~$L^\infty$ convergence of the coefficients,
	see e.g.~\cite{Kato}, one infers that $\t p$ cannot be linearly unstable, hence it is necessarily linearly stable,
	owing to Assumption \ref{ass:multi}.
	But this contradicts the property that the stable, periodic steady states are isolated,
	showed before.
\end{proof}

In order to apply Lemma~\ref{lem:perturb0} in the next subsection, we show that the platforms of any terrace are
	necessarily linearly stable. Indeed, while this is the case of the terrace provided by~\cite{GR_2020}, reclaimed in Theorem~\ref{th:existence}, 
	at this stage of the paper we are yet to prove that there are no other terraces.
\begin{prop}\label{pro:terrace_allstable}
Under Assumption~\ref{ass:multi}, let $\mathcal{T} = ( (q_k)_{0 \leq k \leq K}, (U_k)_{1 \leq k \leq K})$ be a propagating terrace connecting $\overline{p}$ to $0$ in a direction $e \in S^{N-1}$. Then all the platforms $q_k$ are linearly stable, periodic steady states of \eqref{eq:parabolic}.	
\end{prop}
This is an immediate consequence of the fact that all periodic steady states of \eqref{eq:parabolic} are either linearly stable or linearly unstable, by Assumption~\ref{ass:multi}, of the ordering of the speeds involved in a propagating terrace, and of the following lemma.
\begin{lem}\label{lem:counter}
Let $q$ be a linearly unstable periodic steady state of \eqref{eq:parabolic}, and $U_1$, $U_2$ be two pulsating travelling fronts 
satisfying
$$U_1 (-\infty) = p_1 > U_1 (\cdot) > U_1 (+\infty) = q,$$
$$U_2 (-\infty) = q > U_2 (\cdot) > U_2 (+\infty) = p_2,$$
where $p_1$, $p_2$ are periodic steady states with $p_1 > q > p_2$. Then $U_1$, $U_2$ have respectively a positive and a negative speed. 	
\end{lem}
This is closely related to the counter-propagation assumption introduced in~\cite{FZ}; see also~\cite{DHZ17,Ducrot2} for related arguments. For the sake of completeness, we include a short proof here.
\begin{proof}[Proof of Lemma~\ref{lem:counter}]
	We only consider the front $U_1$ and show that its speed, which we denote by $c_1$, is positive. The other case can be handled in the same way up to a simple change of variable.
	
	First of all, notice that if $u(t,x)$ solves~\eqref{eq:parabolic}, then $\tilde{u} (t,x) = u(t,x) - q(x)$ is a solution of the
	equation
\Fi{=g}
	\partial_t \tilde{u} = \dv (A(x) \nabla \tilde{u} ) + g(x,\tilde{u}),\qquad
	\text{with}\quad g(x,u):=\dv ( A(x) \nabla q (x) ) + f (x,q (x)+u ).
\Ff
The nonlinearity $g$ is still spatially periodic and of class $C^1$. Moreover, $0$ is a steady state for the above equation
and, since the linearised operator around it is $\dv (A(x) \nabla \varphi ) + \partial_u f (x, q)$,
i.e. the linearised operator of the original equation around~$q$, it is linearly unstable with periodic 
principal eigenvalue $\lambda^q>0$.
We now replace $g(x,u)$ by 
$$
\tilde g(x,u):=\Big(\partial_u g(x,0)-\frac{\lambda^q}2\Big)u-\gamma u^2,
$$
with $\gamma>0$ large enough so that $\tilde g(x,u)\leq g(x,u)$ for all $x \in \mathbb{R}^N$ and $u \in [0, \max ( \overline{p} - q)]$, which is possible because $g$ is periodic and $C^1$.
%
For this new nonlinearity, the periodic principal eigenvalue of the linearised operator around $0$ is
$\lambda^q/2$, hence $0$ is still linearly unstable, and now $\tilde g$ fulfils the KPP condition: $u\mapsto \tilde g(x,u)/u$ is decreasing.
Summing up, we have that any solution $\tilde u$ to~\eqref{=g} with $0 \leq \tilde{u} \leq \overline{p} -q$
is a supersolution of a KPP-type periodic equation for which $0$ is unstable.	
	Therefore, according to~\cite{BHR1}, if $\tilde u\not\equiv0$ then one has
	$$\inf_{x \in\mathbb{R}^N} \Big(\liminf_{t \to +\infty} \tilde{u} (t,x)\Big) >0.$$
	Applying this to $\tilde u(t,x)=U_1 (x , x \cdot e - c_1 t ) - q (x)$ yields
		$$\inf_{x \in\mathbb{R}^N} \Big(\liminf_{t \to +\infty}   U_1 (x , x \cdot e - c_1 t ) - q (x) \Big) >0.$$
		Due to $U_1 (\cdot , +\infty) = q$, we reach the wanted conclusion that necessarily $c_1 >0$.
			\end{proof}

Finally, in our proof of uniqueness, we need to handle (and eventually rule out) the possibility of a terrace 
	having some fronts with zero speed. To do this, we show that the limit conditions on the fronts in Definition~\ref{def:puls} hold
	true in a stronger sense.
\begin{prop}\label{prop:cv}
	Let $u(t,x) := U (x, x\cdot e - ct)$ be a pulsating travelling front connecting $q_1 > q_2$ with speed $c \in \mathbb{R}$ in a direction $e \in S^{N-1}$. Then $u(t,x)-q_1(x)\to0$, resp. $u(t,x)-q_2(x)\to0$, locally uniformly in $t$, as $x\cdot e \to -\infty$, resp. $x\cdot e \to +\infty$. 
\end{prop}
\begin{proof}
By parabolic estimates, we know that the entire in time solution $u$ of \eqref{eq:parabolic} is (at least) uniformly continuous. 
On the one hand, if $c \neq 0$, due to the invertible change of variables, $U_1$ inherits the same regularity. 
Therefore, in such case, due to the periodicity in the first variable, 
the convergence of $U(x, z)$ to $q_1 (x)$ (resp. $q_2 (x)$) as $z \to - \infty$ (resp. $z \to +\infty$) is uniform in $x\in\R^N$.
The conclusion of the proposition then follows when $c \neq 0$. 

On the other hand, when $c= 0$, we instead proceed by contradiction and assume that there exists a sequence $x_n$ such that $x_n \cdot e \to -\infty$ yet
$$
	\liminf_{n \to +\infty}  \left( U(x_n,x_n\.e) - q_1 (x_n) \right) >0 .
$$
Let us write $x_n = \xi_n + h_n$ with $h_n \in \mathbb{Z}^N$ and $\xi_n \in [0,1)^N$. In particular, $h_n \cdot e \to -\infty$ and 
$\xi_n \to \bar\xi \in [0,1]^N$, up to extraction of a subsequence. 
Since $U (x, x\cdot e)$ is uniformly continuous by elliptic regularity, we infer
$$\liminf_{n \to +\infty}  \left( U(\bar\xi + h_n,(\bar\xi + h_n)\.e) - q_1 (\bar\xi + h_n) \right) >0 ,$$
whence, recalling that $U$ and $q_1$ are spatially periodic,
$$\liminf_{n \to +\infty}  \left( U(\bar\xi,(\bar\xi + h_n)\.e) - q_1 (\bar\xi) \right) >0.$$
This contradicts the definition of the pulsating travelling front and more precisely the (pointwise) asymptotic. 
%
The convergence as $x \cdot e \to + \infty$ can be dealt with similarly.\end{proof}

\subsection{Uniqueness of the propagating terrace}\label{sec:unique}

Our aim is to show that any pair of terraces 
coincide, up to temporal translations of their fronts.
We will prove this by starting from the lowest fronts of the terraces, and 
then by an iteration argument. 
%
%
For this reason, we consider more generally a terrace $\mathcal{T} = ((q_k)_{0 \leq k \leq K}, (U_k)_{1 \leq k \leq K})$
connecting $\overline{p}$ to some lower state $0\leq \ul p<\ol p$
in a given direction~$e$. 
According to Proposition~\ref{pro:terrace_allstable}, all the $q_k$ are linearly stable.
We let $\varphi_k$ denote the periodic principal eigenfunction 
of the linearised operator around the steady state $q_k$, i.e.~satisfying~\eqref{periodicpe} with $p=q_k$,
normalised by
$\max \varphi_k = 1$.

First, we perturb the profiles of this propagating terrace by increasing 
their limits 
as $z\to\pm\infty$. 
Thanks to the stability of the platforms, we will make use of Lemma~\ref{lem:perturb0} to get that the perturbed profiles are supersolutions in some neighbourhoods of their limit states. The perturbation is performed by 
using the principal eigenfunctions~$\varphi_k$, as well as a cut-off function~$\chi$
which is nondecreasing, smooth and  satisfy
\begin{equation}\label{eq:chi0}
\chi(z)=
\begin{cases}
0 & \mbox{ if } z \leq -1 , \\
1 & \mbox{ if } z \geq 1.
\end{cases}
\end{equation}
\begin{lem}\label{lem:perturb}
With the above notation, for $k\in\{1,\dots,K\}$ define
	\begin{eqnarray*}
	\ol U_k^\eta (x,z) := U_k(x,z)+3\eta (1-\chi(z))\vp_{k-1} (x)	+\eta\chi(z)\vp_k (x), \vspace{3pt}
	\end{eqnarray*}
	There exists $\eta_k >0$ 
	such that, for any 
    $\eta\in(0,\eta_k]$, the function
	$$\ol u_k^\eta (t,x)  :=\ol U_k^\eta(x,x\cdot e-c_kt)$$ 
	satisfies
	$$
	\partial_t \ol u_k^\eta \geq \dv (A(x) \nabla \ol u_k^\eta ) + f (x,\ol u_k^\eta) + \eta^2,
	$$
	whenever
	$$\ol u_k^\eta\geq q_{k-1}(x)-\eta_k,\quad\text{or}\quad
	\ol u_k^\eta\leq q_k(x)+\eta_k. $$
\end{lem}

\begin{proof}
%
Recall that the function $U_k(x,x\.e-c_kt)$ is continuous, converges to $q_{k-1}$ or~$q_k$ as $x\.e-c_kt\to-\infty$ or $+\infty$ respectively,
and moreover it satisfies $q_k<U<q_{k-1}$.
Let $\delta$ be the minimum of the two quantities $\delta$ provided by Lemma~\ref{lem:perturb0}
in the case $p=q_{k-1}$ and $p=q_k$ respectively. We can take $\eta_k\in(0,
\delta/3)$
small enough to ensure the following:
\begin{eqnarray}
	\label{U<}
\forall (t,x)\in\R\times\R^N \text{ \ such that \ } x\.e-c_kt\geq-1,\quad & U_k(x,x\.e-c_kt)<q_{k-1} (x)-4 \eta_k,
\\[5pt]
\label{U>}
\forall (t,x)\in\R\times\R^N \text{ \ such that \ } x\.e-c_kt\leq1,\quad & U_k(x,x\.e-c_kt)>q_k (x)+4 \eta_k.
\end{eqnarray}


For $0<\eta\leq \eta_k$, define $\ol U_k^\eta$ and $\ol u_k^\eta$ as in the statement of the lemma.
Consider $(t,x)$ such that 
$\ol u_k^\eta (t,x) \geq q_{k-1} (x) - \eta_k$.
It follows that 
$$U_k(x,x \cdot e - c_k t)\geq q_{k-1} (x) - \eta_k  -3\eta \geq q_{k-1} (x) - 4\eta_k.$$
We deduce from one hand by~\eqref{U<} that 
$x \cdot e - c_k t < -1$, whence $\ol u_k^\eta (t,x)=
U_k(x,x \cdot e - c_k t)+3\eta\vp_{k-1} (x)$. 
From the other hand, recalling that $3 \eta \leq 3\eta_k \leq \delta$,
we can apply Lemma~\ref{lem:perturb0} with $\sigma=0$ and infer
$$
\partial_t \ol u_k^\eta \geq \dv (A(x) \nabla \ol u_k^\eta ) + f (x,\ol u_k^\eta) + 9\eta^2.
$$
Similarly, when
$\ol u_k^\eta (t,x) \leq q_k (x) +  \eta_k$ one has 
$U_k(x,x \cdot e - c_k t)\leq q_k (x) +  \eta_k$, hence~\eqref{U>} yields $x \cdot e - c_k t>1$
and eventually, by Lemma~\ref{lem:perturb0},
$$
\partial_t \ol u_k^\eta \geq \dv (A(x) \nabla \ol u_k^\eta ) + f (x,\ol u_k^\eta) + \eta^2.
$$
The lemma is proved.
\end{proof}

We are now in a position to prove the uniqueness result for the lowest front of the terrace. In the sequel, we will apply it to terraces connecting $\ol p$ to some lower stable state $\ul p$,
hence we state it in such a general form (with the straightforward adaptation of the definition of the terrace).
\begin{prop}\label{pro:unique_lowest}
	Under Assumption~\ref{ass:multi}, let $\underline{p}$ be a linearly stable, periodic steady state satisfying
	$0\leq \ul p<\ol p$. Let
	$\mathcal{T} = ((q_k)_{0 \leq k \leq K}, (U_k)_{1 \leq k \leq K})$
	and $\mathcal{T'}= ((q'_k)_{0 \leq k \leq K'}, (U'_k)_{1 \leq k \leq K'})$
	be two terraces connecting $\overline{p}$ to $\ul p$ in a direction $e \in S^{N-1}$,
	and let $(c_k)_{1 \leq k \leq K}$ and $(c_k')_{1 \leq k \leq K'}$ be the corresponding speeds. 
	Suppose that  at least one among $c_K$ and $c_{K'}'$ is non-zero.
	Then the lowest profiles $U_K(x,z)$ and $U'_{K'}(x,z)$ 
coincide, up to a shift in the $z$ variable. 
	In particular, $c_K = c_{K'} '$ and $q_{K-1} = q'_{K' - 1}$.
\end{prop}
\begin{proof}
We can assume without loss of generality that $c_{K'} ' \geq c_K$.
We will employ a sliding method
between $U'_{K'}$ and the family 
$(\overline{U}_k^\eta)_{1 \leq k \leq K}$ given by Lemma~\ref{lem:perturb}
as perturbations of the profiles $(U_k)_{1 \leq k \leq K}$.
More precisely, our aim will be to ``glue'' together the $\overline{U}_k^\eta$ (or equivalently, the $\overline{u}_k^\eta$) into a single function acting as a supersolution (in suitable subdomains), then shift it so that, roughly speaking, it lies above $U_{K'} '$, and also at some point in time and space the difference is less than $\eta$. By a limiting argument as $\eta \to 0$ and a strong maximum principle, we will eventually conclude that~$U_{K'}'$ coincides with one of the $\overline{u}_k^0$, which must necessarily be $U_K$ due to the fronts' asymptotics.

\Step{1}{``Gluing'' together the functions $\overline{u}_k^\eta$.}
%
Consider the quantities $(\eta_k)_{1 \leq k \leq K}$ provided by Lemma~\ref{lem:perturb} and call
$$\bar\eta:=\min\{\eta_1,\dots,\eta_K\}.$$
Recall that the platforms $(q_k)_{1 \leq k \leq K}$ of the terrace are strictly ordered by definition. 
Hence, up to decreasing $\bar\eta$ if need be, we may assume~that
\Fi{e<<}
\forall k\in\{1,\dots,K-1\},\quad
q_{k+1}+ 2\bar\eta <q_k.
\Ff
Then, for $1\leq k\leq K$ and $0<\eta<\bar\eta$, consider the functions 
$\ol u_k^\eta (t,x)=\ol U_k^\eta(x,x\cdot e-c_kt)$
provided by Lemma~\ref{lem:perturb}. These functions satisfy
$$\lim_{x\.e-c_kt\to-\infty}\big(\ol u_k^\eta (t,x)-q_{k-1}(x)-3\eta\varphi_{k-1}(x)\big)=0,\qquad
\lim_{x\.e-c_kt\to+\infty}\big(\ol u_k^\eta (t,x)-q_{k}(x)-\eta\varphi_{k}(x)\big)=0.$$
Then, for any $0<\eta<\bar\eta$, there exists $z_\eta\geq2$ such that,
for all $k\in\{1,\dots,K\}$, one has
\begin{eqnarray*}
	\forall (t,x)\in\R\times\R^N \text{ \ such that \ } x\.e-c_kt\leq-z_\eta+1,\quad & 
	\ol{u}_k^\eta(t,x)
>q_{k-1} (x) + 2 \eta \vp_{k-1} (x),\\[5pt]
	\forall (t,x)\in\R\times\R^N \text{ \ such that \ } x\.e-c_kt\geq z_\eta-1,\quad & 
	\ol{u}_k^\eta(t,x)
<q_k (x)+2\eta \vp_k (x).
\end{eqnarray*}
Let us pick $Z_\eta\in\Z^N$ such that $Z_\eta\. e\geq z_\eta$. Then we define
the following translation of $\ol u_k^\eta$:
\Fi{ukeps}
\t u_k^\eta(t,x):=\ol u_k^\eta(t,x-2(k-1)Z_\eta).
\Ff
The property of being a (strict) supersolution to~\eqref{eq:parabolic} 
in the neighbourhoods of $q_{k-1}$ and $q_k$,
granted by Lemma~\ref{lem:perturb}, holds true for the function $\t u_k^\eta$.
We glue together the $(\overline{u}_k^\eta)_{1 \leq k \leq K}$ by setting, for $t\geq0$,
\[
\ol u^\eta(t,x):=\begin{cases}
\t u_1^\eta(t,x) & \text{if }x\. e< c_1t+Z_\eta\.e ,\\[5pt]
\min\big\{\t u^\eta_k(t,x),\t u^\eta_{k+1}(t,x)\big\}
\ \quad&\text{if }\;c_kt+(2k-1)Z_\eta\.e\leq x\. e< c_{k+1}t+(2k+1)Z_\eta\.e\\
&\text{for some $k\in\{1,\dots,K-1\}$,}\\[5pt]
\t u^\eta_K(t,x) & \text{if } x\. e\geq c_K t+(2K-1)Z_\eta\.e .
\end{cases}\]
%
Observe that the ranges of $x\.e$ in the above definition cover the whole $\R$.
The resulting function $\overline{u}^\eta$ is illustrated in Figure~\ref{fig:supersol1}.

The function~$\ol u^\eta$ connects $\ol p+ 3\eta\vp_0$ to $\ul p+\eta\vp_K$ in
the direction $e$. 
Furthermore, on the one hand, for $t \geq 0$ and $x$ inside the strip $c_kt+(2k-1)Z_\eta\.e\leq x\. e\leq c_kt+(2k-1)Z_\eta\.e+1$,
one deduces from~\eqref{ukeps}, the fact that $Z_\eta\. e\geq z_\eta$ and the inequality
$c_{k+1}\geq c_k$, that
$$\t u^\eta_k(t,x)<q_k(x)+2\eta\vp_k (x)<
\t u^\eta_{k+1}(t,x),$$
hence $\ol u^\eta(t,x)=\t u^\eta_k(t,x)$ there.
On the other hand, for $t\geq0$ and $x$ inside the strip 
$c_{k+1}t+{(2k+1)Z_\eta\.e-1}\leq x\. e\leq c_{k+1}t+(2k+1)Z_\eta\.e $,
using also~\eqref{e<<} we find that 
$$\t u^\eta_{k+1}(t,x)<q_{k+1}(x)+2\eta\varphi_{k+1}(x)<q_k(x),$$
which is always larger than $\t u^\eta_k(t,x)$,
hence $\ol u^\eta(t,x)=\t u^\eta_{k+1}(t,x)$. 
This ensures that the function~$\ol u^\eta$ is continuous. 
\begin{figure}
\begin{center}
\includegraphics[width=0.7\textwidth]{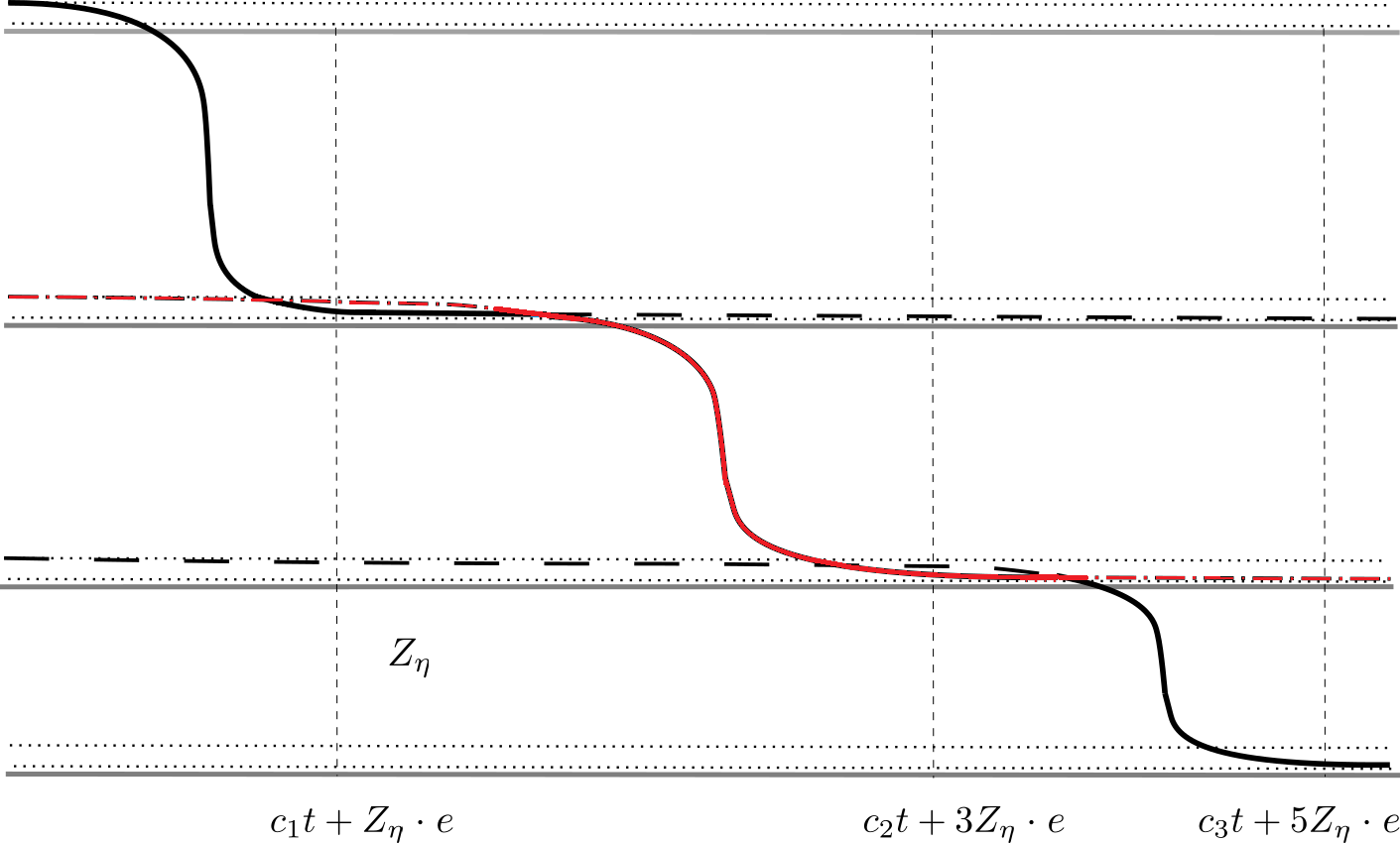}
\end{center}
\caption{Illustration of the construction of the function~$\overline{u}^\eta$ at a fixed time~$t$. 
	The horizontal axis represents the variable $x \cdot e$. 
	The horizontal straight lines are the steady states (including three platforms). The solid line is~$\overline{u}^\eta$,
	 which is taken in each interval as a minimum between two perturbed travelling fronts (drawn in either dashed or dash-dotted~lines).}
\label{fig:supersol1}
\end{figure}
It also entails that, 
for any $t\geq0$ and $x\in\R^N$, there exists $k\in\{1,\dots,K\}$ such that~$\ol u^\eta$ touches~$\t u^\eta_k$ from below at $(t,x)$, in the sense that
\Fi{touch}
0 =\big(\ol u^\eta-\t u^\eta_k\big)(t,x)=
\max_{|y-x|+|c_k(s-t)|<1}
\big(\ol u^\eta-\t u^\eta_k\big)(s,y).
\Ff
%
This property implies that $\ol u^\eta$ acts as a supersolution in the region
where the corresponding $\t u^\eta_k$ is, hence $\ol u^\eta$ will be a {\em generalised supersolution}
whenever sufficiently close to the steady states $q_k$.



\Step{2}{The time-sliding method.}
For $\eta<\bar\eta$, let $\ol u^\eta$ be the function constructed above.
Call 
\begin{equation}\label{def:uprop1}
u(t,x):=U'_{K'}(x,x\. e-c_{K'}' t).
\end{equation}
On the one hand, one has that $u(t,x)<q'_{K'-1}\leq \ol p$ and moreover  $u(t,x)\to\ul p$ as 
$x\.e\to+\infty$, locally uniformly in $t\in\R$.  We point out that $c'_{K'}$ may a priori be zero, and that we made use of Proposition~\ref{prop:cv} here.
On the other hand, the function $\ol u^\eta$ satisfies $\ol u^\eta\geq \ul p+\eta\min\vp_K$
and  $\ol u^\eta(t,x)\to \ol p+ 3\eta\vp_0$ as $x\.e\to-\infty$, locally uniformly with respect to $t\geq0$.
As a consequence, 
we can find $Z\in\Z^N$ such that $Z\.e$ is large enough to have
$$\forall t\in[0,1],\quad
\inf_{x\in\R^N}\big(\ol u^\eta(t,x)-u(t,x+Z)\big)>0.$$
Then, up to replacing $u(t,x)$ with $u(t,x+Z)$, which is still a pulsating travelling front of \eqref{eq:parabolic}, 
we can assume without loss of generality that 
$\inf_{x\in\R^N}(\ol u^\eta(t,x)-u(t,x))>0$ for all $t\in[0,1]$. 

We now distinguish two cases.

\smallskip 
\noindent{\em Case 1: $\ol u^\eta(t_0,x_0)\leq u(t_0,x_0)$ for some $t_0>0$ and $x_0 \in\R^N$.}\\
We can define the following real number:
$$t_\eta:=\inf\big\{t\geq0\ :\ 
\inf_{x\in\R^N}\big(\ol u^\eta-u\big)(t,x)=0\big\}.$$
We know that $t_\eta>1$. There holds that
$$
\forall t\in[0,t_\eta),\
\inf_{x\in\R^N}\big(\ol u^\eta-u\big)(t,x)>0,
$$
and there exists a sequence $(x_n)_{n\in\N}$ in $\R^N$
such that
$$\lim_{n\to+\infty}\big(\ol u^\eta-u\big)(t_\eta,x_n)=0.$$
Up to extraction of a subsequence, there exists $k\in\{1,\dots,K\}$
(depending on~$\eta$) 
such that~\eqref{touch} holds at the points $(t_\eta,x_n)$. We deduce that
\Fi{contact}
\lim_{n\to+\infty}\big(\t u_k^\eta-u\big)(t_\eta,x_n)=0,
\Ff
and
\Fi{above}
\forall s\in[0,t_\eta),\quad
\inf_{|y-x_n|+|c_k(s-t_\eta)|<1}\big(\t u_k^\eta-u\big)(s,y)>0.
\Ff

Let $(h_n)_{n\in\N}$ be a sequence of points in~$ \Z^N$ such that $(\xi_n:=x_n-h_n)_{n\in\N}$ is bounded.
By parabolic estimates, the functions $u(\cdot,h_n+\cdot)$, $\t u_k^\eta(\cdot,h_n+\cdot)$
converge locally uniformly as $n\to+\infty$ (up to subsequences) 
towards two functions $v$ and $v_k^\eta$, the former being a solution to~\eqref{eq:parabolic}
and the latter being a strict supersolution whenever $v_k^\eta>q_{k-1}- \bar\eta$
or $v_k^\eta<q_k+\bar\eta$, thanks to
Lemma~\ref{lem:perturb}.
Moreover, calling $\xi$ the limit of (a subsequence of)
$(\xi_n)_{n\in\N}$, we obtain from~\eqref{contact} and~\eqref{above} that
$$
\min_{|y-\xi|+|c_k(s-t_\eta)|<1}\big(v_k^\eta-v\big)(s,y)=
\big(v_k^\eta-v\big)(t_\eta,\xi)=0.$$
We necessarily have that
$$q_k(\xi)+ \bar\eta \leq v_k^\eta(t_\eta,\xi)\leq q_{k-1}(\xi)-\bar\eta ,$$
because otherwise we would get a contradiction with the parabolic comparison principle. 
We infer that, for $n$ large enough,
\Fi{limxn}
  q_k(x_n)+ \frac{\bar\eta}{2} \leq
\t u_k^\eta(t_\eta,x_n) \leq q_{k-1}(x_n)- \frac{\bar\eta}{2}.
\Ff
%

\smallskip
\noindent{\em Case 2: $\ol u^\eta (t,x)>u(t,x)$ for all $t\geq0$, $x\in\R^N$.}\\
We will reduce to a case similar to the previous one, but in the limit as $t\to+\infty$.
First of all, we restrict ourselves to $\eta<\min\{\bar\eta ,\hat \eta\}$, where $\hat\eta>0$ is such that
$$ q_K+2\hat \eta =\ul p+2 \hat \eta  < q'_{K'-1}.$$
We claim that then necessarily $c'_{K'}=c_{K}$. Suppose by contradiction that this is not the case.
Since by hypothesis $c'_{K'} \geq c_K$, this means that $c'_{K'} > c_K$.
But then, taking $c_K<c<c'_{K'}$, we infer from the definition \eqref{def:uprop1} of $u$ and
that of $\ol u^\eta$, as well as the choice of $\hat \eta$, that
$$\limsup_{t\to+\infty}(\ol u^\eta-u)(t,cte)\leq
\limsup_{t\to+\infty}(q_{K}+\eta\varphi_{K}-q'_{K'-1})(cte)\leq-\hat\eta <0,$$
which is excluded in the present case.
This shows that $c_{K}= c'_{K'}$ in this case, whence, by the hypothesis of the proposition, $c_{K}= c'_{K'}\neq0$.

We now translate $u$ in the $t$ variable. Namely, for $\tau\in\R$, we define
$$D(\tau):=\liminf_{t\to+\infty}\Big(\inf_{x\in\R^N}
(\ol u^\eta(t,x) - u(\tau+t,x)\Big).$$
Since $\ol u^\eta>u$ on $[0,+\infty)\times \R^N$, we have $D(0)\geq0$.
We claim that $D(\tau)\leq0$ for some $\tau\in\R$.
To see this, 
we observe that, for $t\geq0$ and $x_t$ such that 
$x_t\. e=c_K t+(2K-1)Z_\eta\.e$, it holds $\ol u^\eta(t,x_t)=\t u^\eta_K(t,x_t)$
and moreover
$(x_t-2(K-1)Z_\eta)\.e-c_K t=Z_\eta\.e\geq z_\eta$, whence,
by \eqref{ukeps}, 
$$\ol u^\eta(t,x_t)=\t u^\eta_K(t,x_t)<q_K(x_t)+2\eta\vp_K (x_t)<q'_{K'-1}(x_t),
$$
whereas, recalling~\eqref{def:uprop1}, 
$$u(\tau+t,x_t)=U'_{K'}(x_t, -c'_{K'}\tau+(2K-1)Z_\eta).$$
Therefore, if $c'_{K'} \tau >0$ is sufficiently large, one has
$\ol u^\eta(t,x_t) <u(\tau+t,x_t)$ for all $t\geq0$. For such values of $\tau$ we thereby
have $D(\tau)\leq0$.
Since the function $D$ is continuous (because $u$ is uniformly continuous in $t$) 
there exists then $\tau^*\in \mathbb{R}$ such that $D(\tau^*)=0$.

Let $(t_n)_{n\in\N}$ be a sequence diverging to $+\infty$, and $(x_n)_{n\in\N}$ be
a sequence in~$\R^N$ such that
$$\lim_{n\to \infty}\big(\ol u^\eta(t_n,x_n) - u(\tau^*+t_n,x_n)\big)=0.$$
We can find $k\in\{1,\dots,K\}$ (depending on $\eta$) such that~\eqref{touch}
holds at $(t,x)=(t_n,x_n)$ for all $n\in\N$, up to extraction of a subsequence, 
hence
\Fi{contact2}
\lim_{n\to+\infty}\big(\t u_k^\eta(t_n,x_n) - u(\tau^*+t_n,x_n)\big)=0,
\Ff
and in addition, using $D(\tau^*)=0$,
\Fi{above2}
\liminf_{n\to+\infty}\bigg(
\inf_{|y-x_n|+|c_{k}(s-t_n)|<1}\big(\t u_k^\eta(s,y) - u(\tau^*+s,y)\big)\bigg)\geq0.
\Ff

As in the previous case, we consider a sequence 
$(h_n)_{n\in\N}$ of points on~$\Z^N$ such that $(x_n-h_n)_{n\in\N}$ is bounded.
By parabolic estimates, the functions $u(\tau^*+t_n+\cdot,h_n+\cdot)$, $\t u_k^\eta (t_n+\cdot,h_n+\cdot)$
converge locally uniformly, as $n\to+\infty$ (up to subsequences),
towards two functions $v$ and $v_k^\eta$.
The former is still a solution to~\eqref{eq:parabolic}, whereas the latter, owing to Lemma~\ref{lem:perturb},
is a (strict) supersolution whenever $v_k^\eta>q_{k-1}-\bar\eta$
or $v_k^\eta<q_k+\bar\eta$. Moreover,~\eqref{contact2} and~\eqref{above2} imply that 
$v$ touches $v_k^\eta$ from below at~$(0,\xi)$, where $\xi=\lim_{n\to+\infty}(x_n-h_n)$.
The parabolic strong maximum principle then yields 
$$q_k(\xi)+\bar\eta \leq v_k^\eta(0,\xi)\leq q_{k-1}(\xi)-\bar\eta,$$ 
whence, for large enough $n$,
\Fi{limxn2}
q_k(x_n)+ \frac{\bar\eta}{2} \leq \t u_k^\eta(t_n,x_n)\leq q_{k-1}(x_n)- \frac{\bar\eta}{2}.
\Ff

\Step{3}{Conclusion.}
In any case, for $0<\eta<\min\{\bar\eta,\hat\eta\}$, by either~\eqref{contact}-\eqref{above}
or~\eqref{contact2}-\eqref{above2},
we can find $k_\eta \in\{1,\dots K\}$,
$t_\eta>1$, $\tau_\eta \in\R$ and $x_\eta \in\R^N$
such that the following hold:
\Fi{<eta}
\t u_{k_\eta}^\eta(t_\eta,x_\eta)-u(\tau_\eta+t_\eta,x_\eta)<\eta,
\Ff
\Fi{>-eta}
\forall s\in[0,t_\eta),\quad
\inf_{|y-x_\eta|+|c_{k_\eta}(s-t_\eta)|<1}\big(\t u_{k_\eta}^\eta(s,y)-u(\tau_\eta+s,y)\big)>-\eta,
\Ff
and in addition, by~\eqref{limxn} or~\eqref{limxn2},
\Fi{q+d<u<q-d}
q_{k_\eta}(x_\eta)+\frac{\bar\eta}2\leq
\t u_{k_\eta}^\eta(t_\eta,x_\eta)\leq 
q_{{k_\eta}-1}(x_\eta)-\frac{\bar\eta}2.
\Ff

We can find a sequence $\eta \to0^+$ along which all these properties are fulfilled with the same
integer $k_\eta=k$. 
Observe that the associated sequence 
$$x_\eta\.e- 2 (k-1) Z_\eta\.e-c_k t_\eta $$ remains bounded
as $\eta\to 0^+$, owing to~\eqref{ukeps} and~\eqref{q+d<u<q-d}.

Let $(h_\eta)_{\eta}$ in~$\Z^N$ be such that $(\xi_\eta:=x_\eta-h_\eta)_{\eta}$ is bounded, hence
$(h_\eta\.e- 2 (k-1) Z_\eta\.e-c_kt_\eta)_\eta$ is bounded too. Call $\xi$ and $\zeta$ their respective limits as $\eta\to0^+$
(up to subsequences).
We have that, as (a sequence of) $\eta\to0^+$, 
$$\t u_{k}^\eta(t_\eta+t,h_\eta+x)\to U_k(x,x\cdot e-c_kt+\zeta ),\qquad
u(\tau_\eta+t_\eta+t,h_\eta+x)\to U'_{K'}(x,x\cdot e-c'_{K'}t+\zeta'),
$$
locally uniformly in $(t,x)$, with 
$$\zeta':=\lim_{\eta\to0^+}\big(h_\eta\.e-c'_{K'}(\tau_\eta+t_\eta)\big)\in\R\cup\{\pm\infty\}.$$
For clarity we point out that, when $\zeta' = \pm \infty$, then $U'_{K'} (x, x \cdot e - c'_{K'} t + \zeta')$ 
should be understood as one of the periodic states $q'_{K'}$ (i.e. $0$) or $q'_{K'-1}$.

Passing to the limit $\eta\to0$ in \eqref{<eta}-\eqref{>-eta} yields
$$U_k(\xi,\xi\cdot e+\zeta)-U'_{K'}(\xi,\xi\cdot e+\zeta')\leq0,$$
$$\forall  s \in (-1,0),\quad
\inf_{|y -\xi |+|c_{k}s|<1}\big(U_k(y,y\cdot e-c_{k}s+\zeta)-U'_{K'}(y,y\cdot e-c'_{K'}s+\zeta')\big)\geq0.$$
Since
$U_k(x,x\cdot e-c_{k}t+\zeta)$, $U'_{K'}(x,x\cdot e-c'_{K'}t+\zeta')$ are solutions of~\eqref{eq:parabolic}, 
the strong maximum principle and the unique continuation property (see, e.g.,
Section~4 in~\cite{UCparabolic}) imply that 
$$\forall x\in\R^N,\ t\in\R,\quad
U_k(x,x\cdot e-c_{k}t+\zeta)=U'_{K'}(x,x\cdot e-c'_{K'}t+\zeta').
$$
From this, and the properties of \ptf s, one readily deduces that $c_k=c_{K'}'$, 
that~$U'_{K'}$ is necessarily the ``bottom'' of the terrace $\mc{T}$ too,
i.e.~$k = K$, and that $U_K$ and $U'_{K'}$ coincide up to translation (on the
set $(x,x.e)$, $x\in\R^N$, in the case $c_{K'}'=0$).
\end{proof}

The uniqueness of the terrace now follows by recursively applying  Proposition~\ref{pro:unique_lowest}.

\begin{proof}[Proof of Theorem~\ref{th:uniqueness}]
 Let $\mc{T}$ and $\mc{T}'$ be two terraces connecting $\overline{p}$ to $0$ in 
 a direction $e \in \Sph$, with the speeds
of the former being all non-zero. Proposition~\ref{pro:unique_lowest}
  implies that the lowest profiles of~$\mc{T}$ and~$\mc{T}'$ coincide, up to a translation. 
 Dropping this lowest profile from both $\mc{T}$ and $\mc{T}'$, 
we obtain two terraces connecting $\overline{p}$ to the same state $\ul p>0$,
which by Proposition~\ref{pro:terrace_allstable} is thereby stable.
We can then apply Proposition~\ref{pro:unique_lowest} to these new terraces and find that
they also share the same lowest profile (up to shift). Iterating we 
eventually infer that $\mc{T}$ and~$\mc{T}'$ coincide up to shift of their profiles. 
\end{proof}

\begin{rmk}\label{rmk:maybe}
Notice that the argument in this section still partly applies when the terrace~$\mathcal{T}$ has some fronts with zero speeds. Indeed, assume that the lowest (hence fastest) fronts of $\mathcal{T}$ have positive speeds. Then the previous induction still works in the same way, until one reaches a front of~$\mathcal{T}$ with zero speed. Putting this with a symmetrical argument for the uppermost fronts, one would find that any other terrace must include all the fronts of $\mathcal{T}$ whose speeds are not zero. However, terraces may still differ in and between the platforms of~$\mathcal{T}$ connected by travelling fronts with zero speed (see~\cite[Section~6]{MR4130256} for related comments in the one dimensional case). 
\end{rmk}

\subsection{A procedure for the determination of the terrace}\label{sec:algorithm}

Theorem~\ref{th:uniqueness} guarantees the uniqueness of the terrace, in any given direction.
We now present a procedure for determining such a terrace, and in particular its platforms, i.e.~the 
states that are ``selected'' by the terrace. This is based on an iterative argument involving the 
speeds of the fronts connecting the steady states. We also refer to~\cite[Section~6]{MR4130256} for some examples in the one dimensional~case.

We exhibit the procedure in the case where the linearly stable, periodic steady states provided by Assumption~\ref{ass:multi}
are ordered, i.e. they are given by
$$p_0  > p_1 > \cdots > p_M,$$
	for some $M \geq1$.
For any $k\in\{1,\dots,M\}$, the equation is bistable between $p_k$ and $p_{k-1}$, hence 
the propagating terrace provided by Theorem~\ref{th:existence} in any given direction~$e\in\Sph$
reduces to a pulsating travelling front connecting 
$p_{k-1}$ to~$p_k$;
let $U_k$ be its profile and~$c_k$ be its speed.
We point out that the latter is unique by Theorem~\ref{th:uniqueness}
(even when $c_k=0$).
One of the following situations occurs:
\begin{enumerate}[$(i)$]
\item 
the family $(c_k)_{1 \leq k  \leq K}$ is monotonically non-decreasing, i.e.
$$c_{k-1}\leq c_k \quad\text{ for all }\;k\in\{2,\dots,M\};$$
\item
there exists $J\in\{2,\dots,M\}$ such that
$$c_{J-1}>c_J.$$ 
\end{enumerate}
In the case $(i)$, the terrace $\mc{T}$ in the direction $e$ is composed by the whole family of 
front profiles $(U_k)_{k=1,\dots,M}$ (and the corresponding platforms).
In the case $(ii)$, it follows from Theorem~\ref{th:existence} that 
the terrace connecting $p_{J-2}$ to $p_J$ is composed by a single 
pulsating travelling front, with a profile~$\t U_J$ (and it follows from our Theorem~\ref{th:planar_speeds}
that $\t c_J\in[c_J,c_{J-1}]$).
One is therefore left with the family of stacked front profiles
$$U_1,\dots,U_{J-2},\t U_J,U_{J+1},\dots, U_M,$$ 
and thus may repeat the previous argument.
Since the number of fronts is reduced by one, after at most $M$ iterations one ends up in the case $(i)$,
that is, one has constructed the terrace $\mc{T}$.

One should notice that the order in which one picks the pair of fronts to be ``merged together''
does not 
matter, due to the fact that the speed of the resulting front is between the speeds of the two merged fronts.
The general principle is that lowest fronts can only speed down a given front, while highest
fronts can only speed it up.

\subsection{Spreading speeds for planar-like initial data}\label{sec:speed-planar}

In this section, we prove Theorem~\ref{th:planar_speeds}, i.e.~that the 
propagating terrace determines the spreading speeds for planar-like initial data. We will rely on a similar construction as in the proof of Lemma~\ref{lem:perturb},
but this time we will need the function to be a supersolution everywhere.

We place ourselves in the hypotheses of Theorem~\ref{th:planar_speeds}, that is, under Assumption~\ref{ass:multi} and the 
existence of a propagating terrace $\mathcal{T} = ((q_k)_{0 \leq k \leq K}, (U_k)_{1 \leq k \leq K})$
in the direction~$e\in\Sph$ whose speeds are non-zero.
Then Theorem~\ref{th:uniqueness}
guarantees that $\mathcal{T}$ is unique (up to shifts), and in particular
it coincides with the terrace given by Theorem~\ref{th:existence}. The condition~\eqref{ck<>0}
of the speeds being non-zero allows one to strengthen the conclusion of Lemma~\ref{lem:perturb}
by getting a function which is everywhere a supersolution, at the price of slightly increasing its speed.
For its construction, we again make use of a smooth and nondecreasing function $\chi$ satisfying~\eqref{eq:chi0},
as well as of the periodic principal eigenfunctions $\varphi_k$ of the problems~\eqref{periodicpe} with $p=q_k$,
that we normalise by ${\max\vp_k=1}$.

\begin{lem}\label{lem:supersol_planarspeeds}
For $k \in \{1,\dots , K\}$ and $\eps>0$,	call
$$\overline{\psi}_k (t,x) := (1-\chi (x \cdot e - (c_k+ \eps) t)) \varphi_{k-1} (x) + \chi (x \cdot e- (c_k+ \eps) t) \varphi_k (x).$$
Then there exists $\delta_k >0$ such that, 
for all $\eta\in[0,\delta_k]$, the function
\[
\overline{u}_k (t,x) :=  U_k (x, x \cdot e - (c_k+ \eps)t  ) + 
\eta \overline{\psi}_k (t,x)\]
is a supersolution of~\eqref{eq:parabolic}.
\end{lem}

\begin{proof}
%
Consider the quantity $\delta>0$ given by Lemma~\ref{lem:perturb0} in the case $p=q_k$.
We can choose $Z \geq1$ large enough so that the following properties hold:
$$\forall x \in \R^N,
\ \forall z \leq - Z, \qquad | U_k (x,z) - q_{k-1} (x)| \leq  \delta,$$
$$\forall x \in \R^N, \ \forall z \geq Z, 
\qquad | U_k (x,z) - q_{k} (x)| \leq \delta .$$
Observe that $U_k (x , x \cdot e - (c_k + \eps) t  )$
is a supersolution of~\eqref{eq:parabolic} because $U_k (x, z)$ is nonincreasing with respect to $z$,
thanks to Theorem~\ref{th:existence}.
Consider the function $\overline{u}_k$ defined in the statement of the lemma.
On one hand,~since
\[\overline{u}_k (t,x)=
\begin{cases}
U_k (x , x \cdot e - (c_k+ \eps) t  )+ \eta \vp_{k-1}(x)
&
\text{if }\;x \cdot e - (c_k+ \eps) t < -Z ,\\
U_k (x , x \cdot e - (c_k+ \eps) t  )+ \eta \vp_{k}(x)
&
\text{if }\;x \cdot e - (c_k+ \eps) t >Z\,,
\end{cases}\]
Lemma~\ref{lem:perturb0} implies that $\overline{u}_k(t,x)$
is a supersolution of~\eqref{eq:parabolic} in the region $t\geq0$, $|x \cdot e - (c_k+ \eps) t |>Z$,
for any choice of $\eta\in[0,\delta]$. 
%
On the other hand, differentiating in time the equation~\eqref{eq:parabolic} satisfied by $u_k(t,x):=
U_k (x, x\cdot e- c_k t)$, one gets a linear equation for 
$\partial_t u_k(t,x)=-c_k\partial_z U_k$, which we know is a nonnegative or nonpositive function, according to the sign of the (non-zero) speed $c_k$.
It then follows from the parabolic strong maximum principle that $\partial_z U_k$ cannot vanish somewhere
without being identically equal to $0$, which is impossible since $U(\.,-\infty)>U(\.+\infty)$.
As a consequence, 
due to the periodicity of $U_k$ with respect to its first variable, we get
$$\rho := \min_{\substack{x \in \R^N\\ |z| \leq Z}} \left[ -\partial_z U_k (x,z) \right]  >0.$$
We compute
\[
 \partial_t \overline{u}_k - \dv (A(x) \nabla \overline{u}_k) - f (x, \overline{u}_k) \geq
  \eta\big(\partial_t \overline{\psi}_k  -  \dv (A (x) \nabla  \overline{\psi}_k)\big)  
- \eps \partial_z U_k   + f(x,U_k) - f (x,\overline{u}_k) ,
\]
where $U_k$, $\partial_z U_k$ are evaluated at $(x , x \cdot e - (c_k +\eps) t )$. 
The last two terms above are controlled by
$$|f(x,U_k) - f (x,\overline{u}_k)|\leq
\eta\|\partial_uf\|_{L^\infty(\R^N\times[-\eta,\|\overline{p}\|_\infty+\eta])}.$$ 
We then have, for $t \geq 0$ and $|x \cdot e- (c_k +\eps)t| \leq Z$,
\[
 \partial_t \overline{u}_k - \dv (A(x) \nabla \overline{u}_k) - f (x, \overline{u}_k) \geq
\eps\rho-  M_k\eta-\eta\|\partial_uf\|_{L^\infty(\R^N\times[-\eta,\|\overline{p}\|_\infty+\eta])},
\]
with $M_k$ only depending on $A$ and the function $\overline{\psi}_k$, which in turn depends on
$\chi,\vp_{k-1},\vp_k$.
As a consequence, for $\eta$ sufficiently small, the function $\overline{u}_k$ is a supersolution
of~\eqref{eq:parabolic}.
%
\end{proof}

\begin{rmk} 
	Several comments are in order. 
	\begin{enumerate}
		\item
	We point out that analogous versions of Lemmas~\ref{lem:perturb},~\ref{lem:supersol_planarspeeds}
	hold true for subsolutions, just obtained as $\overline{p}(x) - \ol u_k^\eta(t,x)$ and $\overline{p}(x) - \ol u_k(t,x)$
	respectively (which propagate in the direction $-e$ with speed $-c_k$ and $-(c_k+\eps)$ respectively).
	
	\item For the function  $\ol u_k$ of Lemma~\ref{lem:supersol_planarspeeds} to be a supersolution, the strict
	inequality $\partial_z U_k<0$ is necessary, which requires $c_k\neq0$. In the case $c_k = 0$
	it is even unclear whether $z\mapsto U(x,z)$ is continuous, as we discussed in \cite[Remark 3]{GR_2020}.
	Actually, using $\partial_z U_k<0$, one could have constructed a finer supersolution, namely,
	a function approaching exponentially $U_k (x, x \cdot e - (c_k+ \eps)t  )$ as $t\to+\infty$.
	
	\item Using $\ol u_k$ with $k=1$, one can derive the sharp upper bound for the uppermost spreading speed,
	but not for lower speeds.
	To get those, we will ``glue'' together the supersolutions $\ol u_k$. 
This will require that the speeds at which they move
are strictly ordered with respect to $k$. For this purpose, we will perturb the speeds $(c_k)_k$ by 
an increasing family $(\eps_k)_k$.
In the case where the $(c_k)_k$ are already strictly ordered, one may replace the perturbation terms $\eps_k t$ by some exponentially decaying term $C (1 - e^{-\gamma t})$ with some well chosen constants $C,\gamma$, as in the classical ``sandwich argument'' of Fife and Mc Leod~\cite{FMcL}. 
Combined with the above considerations 1 and 2, this would lead to 
a slightly more accurate result under this additional assumption (see also part~$2$ of Remark~\ref{rmk:theo_planar}). 
	
	\end{enumerate}
\end{rmk}

We can now prove the convergence of solutions with ``planar-like'' initial data to the terrace, far from the 
regions $x\.e=c_kt$ where the interfaces are located.

\begin{proof}[Proof of Theorem~\ref{th:planar_speeds}]
	We only derive the upper estimate for the spreading speed. The lower estimate follows by considering the initial datum $\overline{p} - u_0$ (which is still planar-like but in the opposite direction~$-e$) and the corresponding solution of~\eqref{eq:parabolic} with $f(x,u)$ replaced by $f(x, \overline{p}(x) - u)$ (which still satisfies Assumption~\ref{ass:multi}). 
%

Consider an increasing finite sequence of strictly positive numbers $(\eps_k)_{1\leq k\leq K}$.
We have
$$\forall k\in\{1,\dots,K-1\}, \qquad c_k + \eps_k < c_{k+1} + \eps_{k+1}.$$
Next, set $\delta':=\min\{\delta_1,\dots,\delta_K\}$,
where the $\delta_k$ are given by Lemma~\ref{lem:supersol_planarspeeds}.
Call also
$$\delta'':=\min_{1\leq k\leq K-1}\Big(\min(q_k-q_{k+1})\Big),$$
which is positive because the platforms of a terrace are strictly ordered.
Then consider an increasing finite sequence $(\eta_k)_{1\leq k\leq K}$ in $(0,\min\{\delta',\delta''\})$.
Finally, for $k\in\{1,\dots,K\}$,
consider the function $\ol u_k$ given by Lemma~\ref{lem:supersol_planarspeeds}, with $\eta=\eta_k$.
It holds that
$$\overline{u}_k  (t,x) - q_{k} (x) - \eta_k \varphi_k (x) \to 0
\quad\text{ as }\;
x \cdot e - (c_k+ \eps_k) t \to +\infty,$$
while
$$\overline{u}_{k+1} (t,x) - q_{k} (x) - \eta_{k+1} \varphi_k (x) \to 0\quad\text{ as }\;
x \cdot e - (c_{k+1}  +\eps_{k+1}) t \to -\infty.$$ 

Take $\t c_1,\dots,\t c_{K-1} \in\R$ such~that 
$$\forall k\in\{1,\dots,K-1\}, \qquad c_k + \eps_k <\t c_k< c_{k+1} + \eps_{k+1}.$$
Since $\eta_{k+1}>\eta_k$, from the above limits 
we infer the existence of
some large enough $t_k>0$ such that
$$\forall t\geq t_k\ \;\text{and}\; \ x \cdot e = \t c_k t,\qquad
\overline{u}_k (t,x)<\overline{u}_{k+1} (t,x),$$
and in addition, using that $\eta_{k+1}\varphi_{k+1}\leq \eta_{k+1}<\min(q_k-q_{k+1})$,  such that
$$\forall t\geq t_k\ \;\text{and}\; \ x \cdot e \geq \t c_{k+1} t,\qquad
\overline{u}_k (t,x)>\overline{u}_{k+1} (t,x).$$
%
%
%
It follows that the function $\ol u$ defined by
\[
\ol u(t,x):=\begin{cases}
\ol u_1(t,x) & \text{if }x\. e< \t c_1t ,\\[5pt]
\min\big\{\ol u_k(t,x),\ol u_{k+1}(t,x)\big\}
\ \quad&\text{if }\;\t c_kt\leq x\. e< \t c_{k+1}t  \\
&\text{for some $k\in\{1,\dots,K-1\}$,}\\[5pt]
\ol u_K(t,x) & \text{if } x\. e\geq \t c_K t ,
\end{cases}\]
is continuous for $t\geq T:=\max\{t_1,\dots,t_K\}$.
%
The shape of $\ol u$ is similar to the one of $\ol u^\eta$ in the proof
of Proposition~\ref{pro:unique_lowest}, see Figure~\ref{fig:supersol1} for an illustration.


By Lemma~\ref{lem:supersol_planarspeeds}, the function $\overline{u}$ is a {\em generalised supersolution} of \eqref{eq:parabolic} for 
 $t>T$ (being essentially the minimum of supersolutions). 
We also note~that
 $$\forall t\geq t_k\ \;\text{and}\; \ x \cdot e \geq \t c_{k} t,\qquad
\overline{u} (t,x)\leq\overline{u}_k (t,x).$$ 
We deduce that, for all $k\in\{1,\dots,K\}$,
\begin{equation}\label{eq:ing_supersol2}
\limsup_{ t \to +\infty} \Big(\sup_{x \cdot e \geq \t c_k t}\big(\overline{u} (t,x) - q_k (x)\big)\Big)\leq
\limsup_{ t \to +\infty} \Big(\sup_{x \cdot e \geq \t c_k t}\big(\overline{u}_k (t,x) - q_k (x)\big)\Big)
 \leq\eta_k.
\end{equation}
One further sees that
$$\liminf_{x \cdot e \to -\infty} \big(\overline{u} (T,x) -  \overline{p} ( x)\big) >0,$$
$$\inf_{x \in \mathbb{R}^N}  \overline{u} (T,x) >0 .$$
On the other hand, according to the assumptions of the theorem, 
the initial datum $u_0$ satisfies that $0 \leq u_0 \leq \overline{p}$, and 
furthermore
$u_0(x)<\eta$ for $x \cdot e$
sufficiently large, with~$\eta$ belonging to the basin of attraction of $0$.
One deduces that the solution~$u$ of \eqref{eq:parabolic} satisfies, 
on the one hand, $u(t,x)\leq\overline{p} (x)$ for all $t \geq 0$, $x \in \mathbb{R}^N$, by comparison,
and on the other hand, by passing to the limit as $x \cdot e \to + \infty$ and using parabolic estimates,~that
$$\lim_{t \to +\infty} \Big(\, \limsup_{x \cdot e \to + \infty} u(t,x) \Big) = 0.$$
Therefore, we can find $T' > 0$ and $j\in\Z^N$ with $j\.e$ large enough in
such a way that
$$\forall x\in\R^N,\quad
u  (T', x) \leq \overline{u} (T,x-j).$$
Applying a comparison principle we find that 
$u (t ,x) \leq \overline{u} (t + T - T ' ,x-j)$ 
for $t \geq T' $ and $x \in \R^N$. In particular, \eqref{eq:ing_supersol2} yields
$$ \forall k\in\{1,\dots,K\},  \quad 
\limsup_{ t \to +\infty} \Big(\sup_{(x-j) \cdot e \geq \t c_k ( t + T -T' )}\big( u (t ,x) - q_k (x)\big)\Big)\leq\eta_k .$$
Recall that $\t c_k$ was arbitrarily taken in $(c_k + \eps_k,c_{k+1} + \eps_{k+1})$,
and that $(\eps_k)_{1\leq k\leq K}$ was an arbitrary increasing sequence of positive numbers
and $(\eta_k)_{1\leq k\leq K}$ was an arbitrary increasing sequence in~$(0,\min\{\delta',\delta''\})$. This shows the second limit in the statement of the theorem.
%
\end{proof}

\section{The compactly supported case under Assumption~\ref{ass:same_shape}}\label{sec:cpct1}


This section is devoted to the proof of Theorem~\ref{th:spread_cpct1}, but also includes several general results that will be used again later. 
For a given direction $e\in\Sph$, we will consider a propagating terrace $\mc{T}^e$
connecting $\overline{p}$ to~$0$ in the direction~$e$, and, calling~$c_1 (e)$
the speed of its uppermost travelling front, we will assume  that
\Fi{c1>0}
c_1(e) >0\quad\text{ for all }\;e \in S^{N-1}.
\Ff
Hence, all speeds of $\mc{T}^e$ are positive and thus by Theorem~\ref{th:uniqueness},
$\mc{T}^e$ is the unique terrace (up to shifts)
in the direction $e$.
From now on, when we refer to condition~\eqref{c1>0}, it will always be understood that~$c_1(e)$ is associated with the (a fortiori unique) terrace~$\mc{T}^e$.

We~start with a preliminary result about the lower semi-continuity of the mapping $e\mapsto c_1(e)$.
The analogue of this result in the bistable case is given in~\cite[Proposition~2.5]{RossiFG}, see also~\cite{Guo_speeds}.

\begin{lem}\label{lem:lsc}
	Under Assumption~\ref{ass:multi} and~\eqref{c1>0}, the mapping $e\mapsto c_1(e)$ is  bounded and lower semi-continuous
	on~$\Sph$, hence in particular
	$$\inf_{e \in S^{N-1}} c_1 (e) >0 .$$
\end{lem}
\begin{proof}
	Let us start with the boundedness of $c_1$. 
This is a standard fact, that we show here for the sake of completeness.
First, due to the $C^1$ regularity and spatial periodicity of $f$, there exists $C>0$ such that 
	$$\forall x \in \mathbb{R^N},\ u \in [0, \max \overline{p}],\quad
	f(x,u) \leq C u.$$
	Then, a direct computation shows that there exists $\overline{c}$ large enough,
	depending on $A$ and $C$, 
	such that the function
	$$ \overline{u} (t,x) := \min \left\{ \overline{p} (x) , e^{- (x \cdot e - \overline{c} t )} \right\} $$
	is a supersolution of \eqref{eq:parabolic}, for any  $e\in\Sph$.
Applying our Theorem~\ref{th:planar_speeds} on the planar case to the initial datum~$\overline{u} (0,x)$, we infer by comparison that $c_1 (e) \leq \overline{c}$. The speed $c_1$ being also positive by~\eqref{c1>0}, the claimed boundedness follows.
	
	Now fix $e\in\Sph$ and consider any sequence $(e_n)_{n\in\N}$ in $\Sph$
	converging to $e$. From the boundedness, up to extraction of a subsequence, $(c_1 (e_n))_{n\in\N}$ admits a limit $c\in [0,+\infty)$. We need to show that $c\geq c_1(e)$.

Let $U^e$ be the uppermost profile of the terrace $\mc{T}^e$.
Then, take $\eta>0$ such that 
$\ol p-\eta$ is in the basin of attraction of $\ol p$, cf.~Lemma~\ref{lem:perturb0},
and moreover $U^e(\.,+\infty)<\ol p-\eta$. We can then find $\zeta^e\in\R$ such that
\Fi{U-p}
\min_{x\in\R^N}(U^e(x,\zeta^e)-\ol p(x))=-\frac\eta2.
\Ff
We then consider the following translations of the \ptf s:
$$u_n(t,x):=U^{e_n}(x,x\.e_n-c_1(e_n)t+\zeta^{e_n}).$$
By parabolic estimates, the $u_n$ converge locally uniformly as $n\to+\infty$
 (up to subsequences) towards a solution $u(t,x)$ of~\eqref{eq:parabolic}. 
%
%
%
Moreover, on the one hand,
for $x\in\R^N$ such that $x\. e<0$, we have that $x\. e_n<0$ for $n$ large enough, whence by~\eqref{U-p},
$$u(0,x)=\lim_{n\to+\infty} U^{e_n}(x,x\.e_n+\zeta^{e_n})
\geq \overline{p}(x)-\frac\eta2.$$
It follows from the parabolic comparison principle that
$$u(t,x) \geq \underline{u} (t,x),$$
for any $t \geq 0$ and $x \in \mathbb{R}^N$, where $\underline{u}$ is the solution of~\eqref{eq:parabolic} together with
$$\underline{u} (0,x) := \left\{
\begin{array}{ll}
\overline{p} (x) - \frac{\eta}{2} & \mbox{if } x \cdot e < 0, \vspace{3pt} \\
0 & \mbox{if } x \cdot e \geq 0 .
\end{array}
\right.
$$
Applying Theorem~\ref{th:planar_speeds} to~$\underline{u}$, we deduce that
$$\lim_{ t \to +\infty} \sup_{x \cdot e \leq (c_1 (e) - \varepsilon) t} | u (t,x) - \overline{p} (x) | = 0,$$
for any $\varepsilon >0$.

On the other hand, for $t\geq0$ and $x\in[0,1]^N$, it holds by~\eqref{U-p} that
\[\begin{split}
	&\min_{x\in[0,1]^N}\Big(u(t,x+cte+\sqrt{N}e)-\ol p(x+cte+\sqrt{N}e)\Big)\\
&=\lim_{n\to+\infty}\min_{x\in[0,1]^N} \Big(U^{e_n}(x+c_1(e_n)t e_n+\sqrt{N}e_n,x\.e_n+\sqrt{N}+\zeta^{e_n})
-\ol p(x+cte+\sqrt{N}e)\Big)\\
&\leq \lim_{n\to+\infty}\min_{x\in[0,1]^N} \Big(U^{e_n}(x+c_1(e_n)t e_n+\sqrt{N}e_n,\zeta^{e_n})-\ol p(x+c_1(e_n)t e_n+\sqrt{N}e_n)\Big)\\
&=-\frac\eta2.
\end{split}\]
From these two facts, we eventually infer that $c\geq c_1(e)$.
\end{proof}

In the following two subsections, we derive an upper and a lower bound on the spreading speed of the level sets,
respectively. 
Whether these bounds coincide in general remains an open problem. We will see that this is the case under Assumption~\ref{ass:same_shape}, i.e.\ when all propagating terraces in all directions share the same platforms. 

\subsection{An upper bound for the spreading speeds}\label{sec:upper}

An estimate from above of the spreading speeds of solutions with compactly supported initial data
would immediately follow from the comparison with solutions starting from planar-like initial data, whose spreading speeds 
have been derived in Section~\ref{sec:speed-planar}. 
However, we need an upper estimate which is
uniform with respect to the direction, and this requires an additional argument.

More precisely, the next result asserts that the spreading shape is contained in a set, that one may recognize 
as the upper estimate stated in both Theorems~\ref{th:spread_cpct1} and~\ref{th:spread_cpct11}.

\begin{prop}\label{spread_upper}
	Under Assumption~\ref{ass:multi} and~\eqref{c1>0}, let~$0\leq p<\ol p$ be a linearly stable, spatially periodic, steady state of~\eqref{eq:parabolic}, and let $W_{c[p]}$ be given by Definitions~\ref{def:W} and~\ref{def:Wp}.
	Then, for any solution $u$ emerging from a compactly supported initial datum $0 \leq u_0 \leq \overline{p}$,
	it holds:
	$$\forall \eps >0, \qquad 
	\limsup_{t \to +\infty}
	\bigg( \sup_{x  \in\, \R^N\setminus (1+\varepsilon)t  W_{c[p]}  } \big(u(t,x) -  p(x)\big)\bigg) \leq 0.$$
\end{prop}
\begin{proof}
	First recall Definitions~\ref{def:W} and~\ref{def:Wp}. By~\eqref{c1>0} and Theorem~\ref{th:uniqueness}, for any $e \in \Sph$, there exists a unique (up to shifts) propagating terrace~$\mathcal{T}^e$ connecting~$\ol p$ to~$0$ in the direction~$e$. Then, for~$0\leq p<\ol p$ a linearly stable, spatially periodic, steady state of~\eqref{eq:parabolic}, there exists a unique profile~$U$ contained in~$\mathcal{T}^e$ such that
		$$U (\cdot, +\infty) \leq p (\cdot),\qquad
		\max_{x\in\R^N}(U (x, -\infty)-p(x)\big)>0.
		$$
		The speed of that front is denoted by $c[p] (e)$ and the associated Wulff shape is
		$$W_{c [p]}:= \bigcap_{e \in S^{N-1}}\big\{ x \in \R^N \ | \ x \cdot e \leq  c[p] (e) \big\}.$$ 
		By construction, $c [p] (e) \geq c_1 (e)$. Then condition~\eqref{c1>0} together with Lemma~\ref{lem:lsc} yield $\inf_{e} c [p] (e) >0$. It follows that $W_{c[p]}$ is a non-empty, convex and compact set, which can equivalently be written as
		\Fi{wpe}
		W_{c[p]} = \{ r e \ | \ e\in S^{N-1} \text{\; and \;}  0 \leq r \leq \ol w_p (e) \}
		\quad\text{ with }\
		\ol w_p (e) := \inf_{\substack{ e' \in S^{N-1}\\ e' \cdot e >0}} \frac{c[p] (e')}{e \cdot e'},
		\Ff
		and also that the function $e\mapsto \ol w_p(e)$ is positive and continuous on $\Sph$, c.f.~\cite[Proposition 2.4]{RossiFG}.
		
	Next, fix some arbitrary $\varepsilon >0$. Since $u_0$ is compactly supported, for any $e \in S^{N-1}$ one can 
	 find another initial datum $\t{u}_0^e$ satisfying the assumptions of Theorem~\ref{th:planar_speeds} and such that $\t{u}_0^e \geq u_0$ in~$\mathbb{R}^N$. The parabolic comparison principle implies that $u$ is smaller than or equal to the solution emerging from $\t{u}_0^e$ and therefore we deduce from the upper estimates in
	Theorem~\ref{th:planar_speeds} that, for any $e \in S^{N-1}$,
\begin{equation}\label{upper_notuniform}
\limsup_{ t \to +\infty}\bigg( \sup_{ x \cdot e \geq (1+\varepsilon) t c[p] (e)} 
	\big(u(t,x)-  p (x)\big)\bigg) \leq 0.
\end{equation}
In order to derive from this the desired estimate,
one would need to show that \eqref{upper_notuniform} holds true uniformly with respect to $e\in\Sph$.
We do so by showing that one can reduce to a finite number of directions $e$, up to replacing $\eps$ with $2\eps$.
Namely, consider~the~set
$$(1+2\eps)W_{c[p]} .$$
Its boundary $(1+2\eps)\partial W_{c[p]} $ is a compact set which does not intersect the compact set $(1+\eps)W_{c[p]} $,
as it is seen from the expression~\eqref{wpe}, recalling from above that 
the function $\ol w_p$ is positive and continuous.
This means, by the definition of $W_{c[p]}$, that the family
$$\big\{ x \in \R^N \ | \ x \cdot e > (1+\eps) c[p] (e)  \big\},\quad e\in\Sph,$$
is a cover of $(1+2\eps)\partial W_{c[p]} $. We extract from it a finite subcover, that is,
$$(1+2\eps)\partial W_{c[p]} \subset
\bigcup_{j=1}^J
\big\{ x \in \R^N \ | \ x \cdot e_j > (1+\eps) c[p] (e_j) \big\}.$$
Therefore, since the set 
$$\bigcap_{j=1}^J
\big\{ x \in \R^N \ | \ x \cdot e_j \leq (1+\eps) c[p] (e_j) \big\}$$
is connected, and intersects $(1+2\eps) W_{c[p]}  $ (at least at the origin) but not $(1+ 2 \eps) \partial W_{c[p]} $, we deduce
$$\bigcap_{j=1}^J
\big\{ x \in \R^N \ | \ x \cdot e_j \leq (1+\eps)c [p] (e_j) \big\}
\subset(1+2\eps)W_{c[p]} .$$
Passing to the complementary on the above inclusion, and using~\eqref{upper_notuniform}
on the (finite number of) directions $e_1,\dots,e_J$, one gets
$$\limsup_{t \to +\infty}
\bigg( \sup_{x  \in\, \R^N\setminus (1+2\varepsilon)t  W_{c[p]}   } \big(u(t,x) -  p(x)\big)\bigg) \leq 0,$$
which is the desired estimate with $\eps$ replaced by $2\eps$.
\end{proof}

\subsection{A lower bound for the speed towards $\overline{p}$}

Let us turn to the lower bound. We first derive the estimate for the speed of spreading 
towards the uppermost steady state $\overline{p}$, 
then we will apply it iteratively to deal with lower states.
\begin{prop}\label{spread_lower_first}
	Under Assumption~\ref{ass:multi} and~\eqref{c1>0},
let $u$ be a solution emerging from an initial datum $0 \leq u_0 \leq \overline{p}$
for which~\eqref{spread_start} holds. Then $u$ satisfies
\begin{equation}\label{lower_spread_first1}
\forall \eps\in(0,1), \qquad 
\liminf_{t \to +\infty}
\bigg( \inf_{x  \in  (1- \varepsilon)t   W_{c_1}  } \big(u(t,x) -  \overline{p}(x)\big)\bigg) \geq 0,
\end{equation}
where, consistently with the first item of Definition~\ref{def:W},
\Fi{W1}
W_{c_1} := \bigcap_{e \in S^{N-1}}
\{ x \in \R^N \ | \ x \cdot e \leq c_1 (e) \}.
\Ff
\end{prop}
We emphasise that Proposition~\ref{spread_lower_first} only requires Assumption~\ref{ass:multi} and the positivity of the speeds of the terraces. We 
remind that the latter ensures that the terrace is unique (up to translation of its profiles) in any given 
direction, owing to Theorem~\ref{th:uniqueness}.
However, Proposition~\ref{spread_lower_first} only concerns the uppermost state of the terrace.
In the sequel we will apply it to terraces connecting some lower positive state to 0, then
Assumptions~\ref{ass:same_shape} or~\ref{ass:tangentialC1} will be crucial to obtain the sharp estimate on the spreading speeds. For instance, Assumption~\ref{ass:same_shape} ensures that for any $1 \leq k \leq K-1$, the propagating terrace connecting the platform $q_{k}$ to $0$ is a subset of the original terrace connecting $\ol p$ to $0$. Hence the uppermost front of the former is 
just the $k$-th front of the latter, regardless of the direction.

We prove Proposition~\ref{spread_lower_first} by combining and refining the 
geometric arguments employed in~\cite{RossiFG} and \cite{Holes} to derive respectively
the Freidlin-G\"artner formula  
and the sufficient condition for invasion
(i.e.~the local convergence to $\ol p$) in the bistable case. 
These arguments make a link between compactly supported initial data and
planar data, hence will allow us to use the spreading speeds for the latter
already established in Theorem~\ref{th:planar_speeds}.

For later use, we rewrite  the set~$W_{c_1}$ as follows:
\Fi{W1bis}
W_{c_1}:= \{ r e  \ | \ e \in S^{N-1} \text{ and } 0 \leq r \leq w_1 (e)  \}
\quad\text{ with }\
w_1 (e) := \inf_{\substack{ e' \in S^{N-1}\\ e' \cdot e >0}} \frac{c_1 (e')}{e'\cdot e},
\Ff
which is possible because the function $e \mapsto c_1(e)$ is positive, by~\eqref{c1>0}.

We start by showing that the estimate in~\eqref{lower_spread_first1} holds for a family of solutions 
indexed by $\eps$, 
then we will 
deduce the estimate for arbitrary $\eps>0$ for any function fulfilling~\eqref{spread_start} by a comparison argument.
\begin{lem}\label{lem:spread_ini}
Under Assumption~\ref{ass:multi} and~\eqref{c1>0}, for any
$\eps\in(0,1)$,
there exists a solution~$u^\eps$ with
a compactly supported, continuous initial datum $0 \leq u^\eps_0 < \overline{p}$ 
such~that
\Fi{utop}
\lim_{t \to +\infty} 
\bigg(\sup_{x \in (1-\varepsilon) t W_{c_1}} \big| u^\eps(t,x) - \ol p (x)\big|\bigg) = 0,
\Ff
with $W_{c_1}$ given equivalently by~\eqref{W1} or~\eqref{W1bis}. 
\end{lem}
\begin{proof}
	The proof is carried out in several steps. Throughout the proof, $\eps\in(0,1)$ is fixed.

\Step{1}{A sufficient condition.}
First of all, we know that the mapping $\Sph\ni e \mapsto c_1 (e)$ is positive and
lower semi-continuous, by Lemma~\ref{lem:lsc}.
This implies that $\min_e c_1(e)>0$ and therefore we deduce from \cite[Proposition 2.4]{RossiFG}
that the mapping $\Sph\ni e \mapsto w_1(e)$ defined in~\eqref{W1bis} is positive and continuous.
Because of this, we can find
a smooth function $\widetilde{w} : S^{N-1} \to \mathbb{R}$ satisfying
$$( 1 - \varepsilon ) w_1 <  \widetilde{w} \leq 
\Big( 1 - \frac{\varepsilon}2 \Big) w_1. $$
We then call
$$\widetilde{W} : = \{ re \ | \ e \in S^{N-1} \text{ and } 0 \leq r \leq \widetilde{w} (e) \}.$$
This is a smooth compact set which is star-shaped with respect to the origin and satisfies
$$( 1 - \varepsilon ) W_{c_1} \subset \inter(\widetilde{W})\quad\text{and}\quad
\widetilde{W} \subset 
\Big( 1 - \frac{\varepsilon}2 \Big) W_{c_1}.$$
Next, we consider the quantity $\delta>0$ provided by Lemma~\ref{lem:perturb0}, and  
$\vp_0$ the periodic principal eigenfunction
of the linearised operator around $\ol p$. We then set $\eta:=\delta\min\varphi_0$ and define  the function
$$\ul u:=\ol p - \eta\vp_0.$$ 
By Lemma~\ref{lem:perturb0}, this is a strict
subsolution to~\eqref{eq:parabolic} which lies in the basin of attraction of~$\ol p$.
We claim that if a solution $u^\eps$ to~\eqref{eq:parabolic} fulfils 
\Fi{u>ul}
\forall t\geq0,\ \forall x\in t\,\widetilde{W},\quad
u^\eps (t,x)>\ul u(x),
\Ff
then it also satisfies the estimate \eqref{utop}.
Consider indeed an arbitrary positive sequence~$(t_j)_{j\in\N}$ diverging to $+\infty$, and let~$(x_j)_{j \in \N}$ be
such that $x_j$ is a maximising point for $| u^\eps(t,\.) - \ol p|$ 
on the compact set $t_j(1-\eps) W_{c_1}$.
Let then $(h_j)_{j\in\N}$ in~$\Z^N$ be such that $(x_j - h_j)_{j \in \N}$ is bounded. 
Since $x_j/t_j\in(1-\varepsilon)W_{c_1}$, which is a compact contained in the interior of $\t W$, we have
that, for any given~$(t,x)\in\R\times\R^N$, $\frac{h_j+x}{t_j+t}\in\t W$ for $j$ sufficiently large.
As a consequence, if~\eqref{u>ul} holds, it implies that $u^\eps(t_j+t,h_j+x)>\ul u(x)$ for such values of $j$. Therefore $u^\eps(t_j+\.,h_j+\.)$ converges
(up to subsequences) locally uniformly to an entire in time solution~$u_\infty$ 
of~\eqref{eq:parabolic}  satisfying 
$$\forall(t,x)\in\R\times\R^N,\quad
\ul u(x)\leq u_\infty(t,x)\leq\ol p(x).$$
Since~$\ul u$ belongs to the basin of attraction of~$\ol p$,
one readily deduces by comparison that~$u_\infty\equiv\ol p$.
We have thereby shown that property~\eqref{u>ul} is sufficient to have~\eqref{utop}.

\Step{2}{The contradictory assumption.}
Consider a sequence of solutions~$(u_n)_{n\in\N}$ to~\eqref{eq:parabolic} whose initial data are compactly supported, continuous, fulfil~$0\leq u_n(0,\cdot)<\ol p$ and moreover
$$\forall x\in B_n,\quad
u_n(0,x)\geq \ol p(x)-\frac1n.$$
Here $B_n$ denotes the ball of radius $n$ centred at the origin. The following property is readily deduced from parabolic estimates: 
\Fi{un->p}
\lim_{n\to+\infty}u_n(t,x)=\ol p(x) \quad\text{locally uniformly in $(t,x)\in[0,+\infty)\times\R^n$}.
\Ff
We claim that
the function $u_n$ fulfils~\eqref{u>ul} for $n$ sufficiently large, depending on $\eps$,
hence by the previous step it satisfies~\eqref{utop}.
%
%
%
%
%
%
%
Assume by contradiction that this is not the case. Then, for any $n\in\N$, 
it holds that
$$t_n:=\inf\{t\geq0\ :\  \exists x\in t\,\widetilde{W},\ u_n(t,x)\leq\ul u(x)\}\,<+\infty.$$
We know from \eqref{un->p} that $t_n\to+\infty$ as $n\to+\infty$.
In particular, $t_n>0$ for $n$ sufficiently large and it follows
from the definition of $t_n$ that 
\Fi{before-contact}
	\forall t\in[0,t_n), \ \forall x\in t\,\widetilde{W},\quad
	u_n(t,x)>\ul u(x),
\Ff
and, moreover, being $t_n \widetilde{W} $ compact,
that there exists $x_n\in t_n \widetilde{W} $ such that
$$u_n(t_n,x_n)=\ul u(x_n).$$
Recall that $\ul u$ is a strict subsolution, hence the parabolic strong maximum principle
necessarily implies that $x_n\in \partial (t_n \widetilde{W} )$, that is, 
$x_n/t_n\in \partial\widetilde{W}$.

Consider now $h_n \in \mathbb{Z}^N$ such that $\xi_n := x_n - h_n \in [0,1)^N$.
We define
$$\t u_n (t,x) := u_n (t_n+t,h_n+x).$$
Up to extraction of a subsequence, the following limits exist:
$$\xi_n \to \xi_\infty \in [0,1]^N,\qquad
x_n/t_n\to \zeta\in \partial\widetilde{W}.$$
Also, always up to subsequences,
by standard parabolic estimates and spatial periodicity of the equation, the functions~$\t u_n$ converge 
to $\t u_\infty$, an entire solution of \eqref{eq:parabolic} which fulfils by construction (and
by periodicity of $\ul u$)
\Fi{uinfty=}
\t u_\infty (0,\xi_\infty) = \ul u(\xi_\infty).
\Ff
Moreover,~\eqref{before-contact} rewrites for the $\t u_n$ as
\Fi{before-contact-bis}
\forall t\in[-t_n,0), \ \forall x\in (t_n+t) \widetilde{W}-\{x_n\},\quad
\t u_n(t,x+\xi_n)>\ul u(x+\xi_n).
\Ff
We assert that this entails 
\begin{equation}\label{eq:limit_shape}
\forall t \leq 0,  \; \; 
\forall x\cdot \nu   \leq \left(1 - \frac{\varepsilon}{3} \right) c_1 (\nu  )\, t, 
\quad \t u_\infty (t,x+\xi_\infty) \geq \ul u(x+\xi_\infty),
\end{equation}
where $\nu$ is the 
outward unit normal vector to $\widetilde{W}$ at the point $\zeta$ and, 
we recall, $c_1 (\nu)$ is the speed of the uppermost front
of the terrace $\mc{T}^\nu$ in the direction $\nu$.

The first crucial observation to derive~\eqref{eq:limit_shape}
is that the $t$-dependent 
sets 
$(t_n+t) \widetilde{W}$ expand at a given boundary point 
$(t_n+ t )\widetilde w (e)e$
with the (positive) constant normal speed $\widetilde w (e)e\.\t\nu$, where $\t\nu$ is the outward normal at that
point, hence $e\.\t\nu>0$.
The second observation is that~$0\in\partial\big(t_n\widetilde{W}-\{x_n\}\big)$, for any $n\in\N$,
and that  the normal at that point converges to~$\nu$.
The last one is that, because $\widetilde{W}$ is compact and smooth, it satisfies 
uniform interior and exterior sphere 
conditions of some radius $\rho>0$ on the boundary, 
whence its dilation $(t_n + t )\widetilde{W}$
fulfils these conditions with radius $(t_n + t )\rho$, which for any~$t$ tends to~$+\infty$ as
$n\to+\infty$. This means that $(t_n + t )\widetilde{W}$ ``flattens'' to a half-space around each of
its boundary points as 
$n\to+\infty$.
These geometric observations are made rigorous in the proof of~\cite[Theorem~2.3]{RossiFG}, 
leading to the conclusion
that, as $n \to +\infty$, the set $(t_n + t )\widetilde{W} - \{ x_n \}$ invades the half-space 
$$ \{ x \cdot \nu <  \widetilde{c} t  \},$$
where
$$\widetilde{c} := \widetilde{w} \left( e \right) e\.\nu ,\qquad
e:=\frac{\zeta}{|\zeta|},$$
and, we recall, $\zeta$ is the limit of $x_n/t_n$ and
 $\nu$ is the outward normal to $\widetilde{W}$ at $\zeta$. 
Therefore, we deduce from~\eqref{before-contact-bis} that
$$\forall t\leq0,\ \forall x \cdot \nu <  \widetilde{c} t,
\quad
\t u_\infty (t,x + \xi_\infty ) \geq \ul u (x+\xi_\infty).$$ 
Now, our choice of $\widetilde{w}$ and the definition~\eqref{W1bis} of $w_1$ imply that
$$\widetilde{c}\leq\Big( 1 - \frac\varepsilon3\Big) w_1(e) e\.\nu\leq \Big( 1 - \frac\varepsilon3\Big)c_1(\nu), $$
and therefore 
\eqref{eq:limit_shape} holds.

\Step{3}{Conclusion.}
Let $h\in \Z^N$ be such that $h\.\nu<0$. For $m\in\N^*$ define 
the time $t_m$ by
\Fi{tm}
\left(1 - \frac{\varepsilon}{3} \right) c_1 (\nu  )\, t_m:=m h\.\nu<0.
\Ff
We find from~\eqref{eq:limit_shape} (and the periodicity of $\ul u$) that
$$\forall x\.\nu\leq0,\quad
\t u_\infty(t_m,x+m h+\xi_\infty)\geq \ul u(x+\xi_\infty).$$
It follows by comparison that 
\Fi{uinfty>v}
\forall t\geq0,\ \forall x\in\R^N,\quad
\t u_\infty(t+t_m,x+m h)\geq v(t,x),
\Ff
where $v$ is the solution to~\eqref{eq:parabolic} emerging from the
initial datum
$$\1_{(-\infty,0]}\big((x-\xi_\infty)\.\nu\big)\,\ul u(x).$$
We recall 
that $\ul u(x)=\ol p - \eta\vp_0\geq \ol p - \eta$, and that $\ol p - \eta\geq \ol p - \delta\varphi_0$
belong to the basin of attraction of~$\ol p$ by the choice of $\delta$.
Therefore, Theorem \ref{th:planar_speeds} applies to~$v$ and yields
$$\lim_{t \to +\infty} 
\bigg(\sup_{ x \cdot \nu \leq \big(1 - \frac\eps3 \big)c_1(\nu) t} \big|v(t,x) - \ol p (x)\big|\bigg)  = 0.$$
Then, in particular, recalling the definition~\eqref{tm} of $t_m$ (which yields
$t_m\to-\infty$ as $m\to+\infty$), we~get
$$\lim_{m\to+\infty}v(-t_m,-m h)=\ol p(0).$$
As a consequence, since by~\eqref{uinfty>v}
$$v(-t_m,-m h)\leq\t u_\infty(0,0),$$
we conclude that $\t u_\infty(0,0)\geq\ol p(0)$, and therefore 
$\t u_\infty(0,0)\equiv\ol p(0)$ by the parabolic strong maximum principle.
This contradicts~\eqref{uinfty=}.
\end{proof}

Lemma \ref{lem:spread_ini} establishes that if~\eqref{c1>0} holds, i.e.~the terraces have positive speeds,
then the invasion property~\eqref{spread_start} holds for some 
compactly supported initial data.
As a matter of fact, the reverse implication is also valid, as we now show.

\begin{prop}\label{pro:spread_ini}
	Under Assumption~\ref{ass:multi}, if there exists a solution with a compactly supported
	initial datum $0 \leq u_0 \leq \overline{p}$ 
		which satisfies~\eqref{spread_start}, then all the speeds of the terraces, in any direction, are~positive.
\end{prop}
	
	\begin{proof}
		Let $u$ satisfy~\eqref{spread_start}, and let $u_0$ be its initial datum.
		On the one hand, since $u_0$ is compactly supported, at any given time $\bar t>0$
		the function $u(\bar t,\.)$ has a Gaussian decay, and moreover it is strictly smaller than $\ol p$,
		due to the strong maximum principle.
		On the other hand, let $U_1(x, x\cdot e - c_1t)$ be the uppermost front of a given terrace 
		in the direction $e$. Then, either $U_1(\.,+\infty)>0$, or $U_1(\.,+\infty)\equiv0$, but then $U_1(x, x\cdot e - c_1 t)$
		decays at most exponentially to $0$ as $x\.e\to+\infty$; this follows from \cite[Theorem 4.1]{Landis},
		applied to the function $U_1(x+c_1te,x\.e)$ (which solves \eqref{eq:parabolic} with an additional drift term $c_1 e\.\nabla$).
		Therefore, in any case, since $U_1(\.,-\infty)\equiv\ol p$, we can find a shift $h\in\Z^N$ such that~$h\.e$ is sufficiently
		large to have $u(\bar t,x+h)\leq U_1(x, x\cdot e - c_1\bar t)$ for all $x\in\R^N$.
		It then follows from the comparison principle that
			\[\lim_{t\to+\infty}U_1 (x, x\cdot e - c_1 t)\geq \lim_{t\to+\infty}u(t,x+h)=\overline{p}(x),\]
		thanks to~\eqref{spread_start}, which in turn yields $c_1>0$.
		%
		%
\end{proof}

%
\begin{proof}[Proof of Proposition~\ref{spread_lower_first}] The result readily follows from Lemma~\ref{lem:spread_ini}. 
	Indeed, for any $\eps\in(0,1)$, considering the initial datum $u^\eps_0$ provided by that lemma, one infers from~\eqref{spread_start} the existence of some large enough $T_\eps>0$ for which 
$$u(T_\eps,\cdot) \geq u^\eps_0.$$
By comparison one obtains
$$\liminf_{t \to +\infty} 
\bigg( \inf_{x  \in  (1- \varepsilon)t   W_{c_1}  } \big(u(T_\eps+t,x) -  \overline{p}(x)\big)\bigg) \geq 0.
$$
This, in turn, yields the estimate~\eqref{lower_spread_first1} with $\eps$ replaced by $2\eps$,
	by simply noticing that $(1-\varepsilon) t W_{c_1}\supset(1-2\varepsilon) (t+T_\eps) W_{c_1}$ if $\eps<1/2$,
	for $t$ sufficiently large,
	because $W_{c_1}$ is star-shaped with respect to the origin.
	\end{proof}

\subsection{Conclusion under Assumption~\ref{ass:same_shape}}



\begin{proof}[Proof of Theorem~\ref{th:spread_cpct1}]
On the one hand, under Assumption~\ref{ass:same_shape}, applying Proposition~\ref{spread_upper} to any platform $p = q_k$, we immediately get that
$$\forall \varepsilon >0, \; \; \forall k\in\{1,\dots, K\}, \quad \limsup_{t \to +\infty}
\bigg( \sup_{x   \in\, \R^N\setminus   (1+\varepsilon)t  W_{ c [q_k]}  } \big(u(t,x) -  q _k (x)\big)\bigg) \leq 0.$$
On the other hand, the restriction of \eqref{eq:parabolic} to solutions taking values between~$q_k$ and~$0$, is still of the multistable type in the sense of Assumption~\ref{ass:multi} (where $\overline{p}$ should then be replaced by $q_k$). Furthermore, the (unique) propagating terrace in direction~$e$ of the resulting problem, i.e.~connecting $q_k$ to $0$, is given by 
$$((q_j)_{k \leq j \leq K}, (U_j^e)_{k+1 \leq j \leq K }).$$
In other words, it is a subset of the terrace of the original problem, whose speeds are all positive. Applying Proposition~\ref{spread_lower_first} (replacing again $\overline{p}$ by $q_k$), we get that 
$$
\forall \eps \in (0,1), \; \; \forall k \in \{ 0, \dots , K-1 \}, \quad \liminf_{t \to +\infty}
\bigg( \inf_{x  \in  (1- \varepsilon)t   W_{c_{k+1}}  } \big(u(t,x) - q_k (x)\big)\bigg) \geq 0.
$$
Recalling that $c_{k+1} = c[q_{k+1}] $ (see again Definition~\ref{def:Wp}), we have proved both upper and lower estimates in Theorem~\ref{th:spread_cpct1}.
\end{proof}

\section{The case where the Wulff shapes are regular
	}\label{sec:cpct}

As we explained in the Introduction, Assumption~\ref{ass:same_shape} is not always fulfilled. 
We recall that our upper and lower bounds on the spreading shapes, Propositions~\ref{spread_upper} and~\ref{spread_lower_first} respectively,
have been derived without such an assumption. However, while Proposition~\ref{spread_upper} deals with propagation towards a generic stable state, 
Proposition~\ref{spread_lower_first} only deals with the uppermost state $\ol p$, which is not enough
to prove the spreading result, so far unless requiring Assumption~\ref{ass:same_shape}.
For this reason, we will first derive a general lower bound for the spreading shapes towards intermediate stable states.
This will not require any hypothesis other than
Assumption~\ref{ass:multi} and the positivity of the speeds of the terraces.
Then, we will show that this new estimate combines in a sharp way with the upper bound, Proposition~\ref{spread_upper},
and yields Theorem~\ref{th:spread_cpct11}, provided that Assumption~\ref{ass:tangentialC1} on the regularity of some Wulff shapes holds.

\subsection{A lower bound for the speed towards intermediate states}\label{sec:lower}

In order to
derive a lower bound of the set where~$u$ is asymptotically larger than or equal to a given stable
 steady state,
we make use of an iterative argument. This requires the introduction of some more notation, some of which is similar to that used in Assumption~\ref{ass:tangentialC1}. 
It is always understood here that the Assumption~\ref{ass:multi} is in force.

\begin{defi}\label{defi:upsw}
			 For a given 	linearly stable, periodic steady state~$0 < p \leq \ol p$ of~\eqref{eq:parabolic},
			and for a given direction $e \in S^{N-1}$, consider the speed~$c_1^p (e)$ given by Definition~\ref{def:Wp}.
	%
	%
	We call  
	\begin{equation}\label{eq:upsilon_def}
		\Upsilon_p:= 
		\begin{cases}
			\displaystyle W_{c_1^p} =
			\bigcap_{e\in\Sph}\big\{ x \in \R^N \ | \ x \cdot e \leq  c_1^p (e) \big\} & 
			\mbox{ if $ c_1^p (e) > 0$ for all $e \in S^{N-1}$}, \\[12pt]
			\ \ \big\{ 0 \big\} & \mbox{ otherwise,}
		\end{cases}
	\end{equation}
	and then, recursively,
	\Fi{W>pbar}
	\ul W_{\overline{p}}:=\Upsilon_{\overline{p}} ,
	\Ff
	\Fi{W>recursive}
	\ul W_p := \bigcup_{\substack{p' \in E^+ \\ p<p'\leq\ol p \\ \kappa\in[0,1]}}
	\Big(\kappa\, \ul W_{p'} +(1-\kappa)
	\Upsilon_p \Big),
	\Ff
	where $E^+$ denotes the set of linearly stable, positive, periodic steady states.
	\end{defi}
	
Let us explain why those definitions are well posed, 
and in particular why $\Upsilon_p$ does not depend on the choice of the terrace
	$\mathcal{T}_p^e$ in the Definition~\ref{def:Wp} of $c_1^p (e)$.
There are two situations: first, for any $e\in\Sph$, there exists a terrace $\mathcal{T}_p^e$
connecting~$p$ to~$0$ in the direction $e$ whose uppermost speed is positive,
in which case such terrace is unique by Theorem~\ref{th:uniqueness} 
and therefore $c_1^p (e)$ is well defined, and so is~$\Upsilon_p\neq\{0\}$;
second, there exists a terrace $\mathcal{T}_p^e$ in some direction $e$ whose uppermost speed is nonpositive,
hence any other terrace in that direction necessarily shares the same property (otherwise we may again apply Theorem~\ref{th:uniqueness}), whence~$\Upsilon_p=\{0\}$.
%
Notice that the first case in the definition of~$\Upsilon_p$ occurs if and only if
	\begin{equation}\label{hyp:lower0}
	\inf_{e \in S^{N-1}} c_1^p (e) > 0,
	\end{equation}
as a consequence of the lower semi-continuity of $e \mapsto c_1^p (e)$, 
	cf.~Lemma~\ref{lem:lsc} (see also \cite{Guo_speeds,RossiFG}). 
	
	The sets $\Upsilon_p$ are compact and convex. Therefore, the sets $\ul W_p$ 
	are compact because~$E^+$ is finite, cf.~Proposition \ref{pro:finitep}, but they are not convex in general. 
	Since the lower bound on the solutions will be expressed in terms of the sets $\ul W_p$, we make their recursive definition more explicit in the following.

\begin{lem}\label{lem:ulW}
Under Assumption~\ref{ass:multi}, let 
$$p_0\equiv\ol p\,, \ p_1\,,\ \cdots\;, \ p_M\equiv0$$
be the linearly stable, periodic, steady states between~$0$ and $\overline{p}$.
Then, for any $p>0$ among them, it~holds
\Fi{W>}
\ul W_p = \bigcup_{\substack{\kappa_0,\dots,\kappa_{M-1}\geq0,
		\ \sum_{i=0}^{M-1} \kappa_i=1\\
		\kappa_i>0 \text{ only if }p_i\geq p\\
		\kappa_i\kappa_j\neq0 \text{ only if }
		p_i\neq p_j\text{ on }\R^N}}
\ \sum_{i=0}^{M-1} \kappa_i \Upsilon_{p_i}.
\Ff
If in addition the stable states are ordered, that is,
$$p_0 \equiv \overline{p} > p_1 > \cdots > p_M \equiv 0,$$ 
then, for all $k\in\{0,\dots,M-1\}$, it holds
\begin{equation}\label{w_p_conv}
\ul W_{p_k}= \Conv\Bigg(\bigcup_{j=0}^k
\Upsilon_{p_j}\Bigg),
\end{equation}
where $\Conv$ stands for convex hull.
\end{lem}

\begin{proof}
 	Let us prove~\eqref{W>}.
 	For $p=\ol p$, the identity~\eqref{W>} reduces exactly to the definition~\eqref{W>pbar}.
 	Consider now $p\neq\ol p$. We can assume without loss of generality that $p=p_{M-1}$.
	The right-hand side in~\eqref{W>} can be rewritten as
	 $$\bigcup_{\kappa\in[0,1]}\Bigg(
	 \kappa\Upsilon_p\ +\
	 \bigcup_{\substack{\kappa_0,\dots,\kappa_{M-2}\geq0,\ 
	 \ \sum_{i=0}^{M-2} \kappa_i=1-\kappa\\
	 		\kappa_i>0 \text{ only if }p_i>p\\
	 		\kappa_i\kappa_j\neq0 \text{ only if }
	 		p_i\neq p_j\text{ on }\R^N}}
	 \ \sum_{i=1}^{M-2} \kappa_i \Upsilon_{p_i}\Bigg),$$
	 which in turn is equal to
	 	 $$\bigcup_{\kappa\in[0,1]}\Bigg(
	 	 \kappa\Upsilon_p\ +\
	 	 (1-\kappa)\bigcup_{\substack{\kappa_0,\dots,\kappa_{M-2}\geq0,\ 
	 	 		\ \sum_{i=0}^{M-2} \kappa_i=1\\
	 	 		\kappa_i>0 \text{ only if }p_i>p\\
	 	 		\kappa_i\kappa_j\neq0 \text{ only if }
	 	 		p_i\neq p_j\text{ on }\R^N}}
	 	 \ \sum_{i=1}^{M-2} \kappa_i \Upsilon_{p_i}\Bigg).$$
	 Observe that 
	 $$\bigcup_{\substack{\kappa_0,\dots,\kappa_{M-2}\geq0,\ 
	 		\ \sum_{i=0}^{M-2} \kappa_i=1\\
	 		\kappa_i>0 \text{ only if }p_i>p\\
	 		\kappa_i\kappa_j\neq0 \text{ only if }
	 		p_i\neq p_j\text{ on }\R^N}}
	 \ \sum_{i=1}^{M-2} \kappa_i \Upsilon_{p_i}
	 	 	 =
	 	\bigcup_{p'>p}\
	 	\bigcup_{\substack{\kappa_0,\dots,\kappa_{M-2}\geq0,\
	 			\ \sum_{i=0}^{M-2} \kappa_i=1\\
	 			\kappa_i>0 \text{ only if }p_i\geq p'\\
	 			\kappa_i\kappa_j\neq0 \text{ only if }
	 			p_i\neq p_j\text{ on }\R^N}}
	 	\ \sum_{i=1}^{M-2} \kappa_i \Upsilon_{p_i}.$$
	If we assume that~\eqref{W>}
	holds for all stable steady states $p'>p=p_{M-1}$, then the above right-hand side is equal to $\bigcup_{p'>p}\ul W_{p'}$,
	and therefore~\eqref{W>} holds true for $p$, by the definition \eqref{W>recursive}.
	The result is thereby proved by iteration on the number of stable steady states above $p$.
%

Suppose now that the stable states are ordered. 
Then the identity~\eqref{W>} reads, for any $k\in\{0,\dots,M-1\}$,
$$\ul W_{p_k} = \bigcup_{\kappa_0,\dots,\kappa_{k}\geq0,\ 
	\sum_{i=0}^k \kappa_i=1}
\ \sum_{i=0}^{k} \kappa_i \Upsilon_{p_i}.$$
It is a classical fact that, being the sets $\Upsilon_{p_i}$ convex, such a set coincides with the convex hull of 
$\bigcup_{j=0}^k\Upsilon_{p_j}$. Indeed one immediately has
$$\bigcup_{j=0}^k
\Upsilon_{p_j}\subset
\ul W_{p_k} \subset \Conv\Bigg(\bigcup_{j=0}^k
\Upsilon_{p_j}\Bigg),$$
and it remains to show that 
$\ul W_{p_k}$ is convex. For this, take two elements $x,y\in \ul W_{p_k}$, that is,
$$x=\sum_{i=0}^{k} \kappa_ix_i, \quad y=\sum_{i=0}^{k} h_iy_i,
\quad\text{with \ $x_i,y_i\in\Upsilon_{p_i},\quad
\kappa_i,h_i\geq0,\ \ \sum_{i=0}^k \kappa_i=\sum_{i=0}^k h_i=1$,}$$
and $\lambda\in[0,1]$. Then
$$\lambda x+(1-\lambda)y=\sum_{i=0}^{k} [\lambda \kappa_ix_i+(1-\lambda)h_iy_i].$$
Observing that 
$$\lambda \kappa_i x_i+(1-\lambda)h_iy_i=(\lambda \kappa_i+(1-\lambda)h_i)z_i, \quad \text{where \ }
z_i=\frac{\lambda \kappa_i}{\lambda \kappa_i+(1-\lambda)h_i}\,x_i+
\frac{(1-\lambda)h_i}{\lambda \kappa_i+(1-\lambda)h_i}\,y_i,
$$
belongs to $\Upsilon_{p_i}$ by convexity, one eventually deduces
$$\lambda x+(1-\lambda)y=\sum_{i=0}^{k}(\lambda \kappa_i+(1-\lambda)h_i)z_i\in \ul W_{p_k}.$$
%
This concludes the proof of the lemma.
%
\end{proof}

Lemma~\ref{lem:ulW} shows that $\ul W_p$ is composed by the convex hulls of the unions 
of the $(\Upsilon_{p_i})_i$ corresponding to all monotone families of stable
steady states $(p_i)_i$ larger than or equal to $p$.
In the case where the intermediate stable states do not intersect each other (i.e.~they are totally ordered)
the sets~$\ul W_p$ are convex, and they are monotone nonincreasing with respect to $p$, in the sense that the smaller $p$,
the larger~$\ul W_p$.

We can now state our general lower bound for a given stable steady state $p$.
\begin{prop}\label{spread_lower_induc}
Let \eqref{eq:parabolic} be of the multistable type in the sense of Assumption~\ref{ass:multi}, and assume also that
$$c_1^{\ol p} (e) >0\quad\text{ for all }\;e \in S^{N-1}.$$
Let $u$ be a solution emerging from an initial datum $0 \leq u_0 \leq \overline{p}$
for which~\eqref{spread_start} holds. 
Then $u$ satisfies, for any linearly stable steady state $0 < p\leq \ol p$,
\begin{equation*}
\forall \eps \in(0,1), \qquad 
\liminf_{t \to +\infty}
\bigg( \inf_{x  \in  (1- \varepsilon)t   \underline{W}_p  } \big(u(t,x) -  p(x)\big)\bigg) \geq 0,
\end{equation*}
where $\ul W_p$ is equivalently given by the recursion~\eqref{W>pbar}-\eqref{W>recursive} or by~\eqref{W>}.
\end{prop}
Very roughly, Proposition~\ref{spread_lower_induc} states that a lower estimate of the shape of the set where
the solution becomes larger than or equal to~$p$ is obtained by a combination of, on the one hand, the uppermost Wulff shape of the terrace connecting~$p$ to~$0$, and on the other hand, the spreading shapes over all higher steady states.

%
Proposition~\ref{spread_lower_induc} will be proved by induction. Beforehand, we point out that Lemma~\ref{lem:spread_ini} can be extended in the following way.
\begin{lem}\label{lem:spread_ini_bis}
Under the assumptions of Proposition~\ref{spread_lower_induc}, for any linearly stable, periodic steady state $0< p \leq \ol p$ which satisfies \eqref{hyp:lower0}, and any $\eps\in(0,1)$,
there exists a solution $u$ with
a compactly supported, continuous initial datum $0 \leq u_0 < p$ 
such~that
$$
\lim_{t \to +\infty} 
\bigg(\sup_{x \in (1-\varepsilon) t \Upsilon_{p}} \big| u(t,x) -  p (x)\big|\bigg) = 0,
$$
where $\Upsilon_p = W_{c_1^p}$ is given by Definition~\ref{defi:upsw}.
\end{lem}
Indeed, recall that the restriction of \eqref{eq:parabolic} to functions between~$0$ and~$p$ is still of the multistable type, in the sense that Assumption~\ref{ass:multi} holds true if one replaces~$\overline{p}$ with~$p$. 
Then, owing to Lemma~\ref{lem:lsc}, 
Lemma~\ref{lem:spread_ini_bis} is merely a restatement of Lemma~\ref{lem:spread_ini}.

Anyway, from this we will infer the following result.
\begin{lem}\label{lem:W>}
Under the assumptions of Proposition~\ref{spread_lower_induc}, for any
	linearly stable, periodic steady state $0<p\leq\ol p$ and any $\eps\in(0,1)$,
	there exists a solution $u$ with
	a compactly supported, continuous initial datum $0 \leq u_0 < \overline{p}$ 
	such~that
	\Fi{u>p}
	\liminf_{t \to +\infty} 
	\bigg(\inf_{x \in (1-\varepsilon) t\, \ul W_p} \big( u(t,x) - p(x)\big)\bigg) \geq 0.
	\Ff
\end{lem}

Notice that Proposition~\ref{spread_lower_induc} follows from Lemma~\ref{lem:W>}, in a similar way that Proposition~\ref{spread_lower_first} followed from Lemma~\ref{lem:spread_ini} in the previous section. Thus it only remains to show that Lemma~\ref{lem:W>} holds true.
\begin{proof}[Proof of Lemma~\ref{lem:W>}]
	In the case $p=\ol p$, the result is contained in either Lemma~\ref{lem:spread_ini} or Lemma~\ref{lem:spread_ini_bis}. Indeed, observe that~\eqref{W>pbar} and the positivity of $c_1^{\ol p} (e)$ for all $e \in S^{N-1}$, yield~$\ul W_{\overline{p}}= \Upsilon_{\overline{p}} = W_{c_1^{\ol p}}$. We now prove property~\eqref{u>p}
	for any intermediate stable state $0<p<\ol p$ by an iterative argument,
	assuming that it holds for all~$p'>p$.

	First notice that all the sets $\Upsilon_p$ are star-shaped with respect to the origin,
	hence by~\eqref{W>} the same is true for $ \underline{W}_p$.
	It follows from \eqref{W>recursive} that, when $\Upsilon_p = \{ 0\}$, then
		$$ \underline{W}_p =  \bigcup_{\substack{p' \in E^+ \\ p<p'\leq\ol p}}
\ul W_{p'},
$$ 
and the conclusion immediately follows from the inductive hypothesis. So we only consider the other case, 
when \eqref{hyp:lower0} holds and $\Upsilon_p \neq \{0\}$.
	Let $\eps\in(0,1/2)$ and consider a point
	$\zeta\in (1-\varepsilon)\ul W_p$. Owing to~\eqref{W>recursive}, there exist $\kappa \in [0,1]$ and some $p < p' < \ol p$ such that 
	$$\zeta \in (1 - \eps) \left[ \kappa \ul W_{p'} + (1-\kappa) \Upsilon_p \right].$$
	We distinguish three different situations, according to whether $(1-\varepsilon)^{-1} \zeta$ belongs to $\Upsilon_p$, or $\ul W_{p'}$, or none of them.	
	
	\smallskip
{\em Case $(1 -\varepsilon)^{-1} \zeta \in \Upsilon_p $.}\\
Recall also that $\Upsilon_p \neq \{ 0\}$, so that according to~\eqref{eq:upsilon_def} and~\eqref{hyp:lower0}, for any $\lambda>0$, 
		$\lambda\Upsilon_p$ is a convex set containing a neighbourhood of the origin. 	
	One deduces from this that $\zeta$ is in the interior of $(1-\varepsilon/2)\Upsilon_p$,
	hence we can find
	$\delta>0$ such that $B_\delta(\zeta)\subset (1-\eps/2) \Upsilon_p$.
	By Lemma~\ref{lem:spread_ini_bis} there exists a 
	solution~$u^p$ with a compactly supported, continuous initial datum $0 \leq u_0^p < p$ such that
	$$\lim_{t \to +\infty} 
	\bigg(\sup_{x \in (1-\varepsilon/2)t \Upsilon_p} \big| u^p(t,x) - p (x)\big|\bigg) = 0.$$
	We thus derive
	\Fi{uptop}
	\lim_{t \to +\infty} 
	\bigg(\sup_{x \in t B_\delta(\zeta)} \big|u^p(t,x) - p (x)\big|\bigg) = 0.
	\Ff

	\smallskip
{\em Case $(1-\varepsilon)^{-1} \zeta \in   \ul W_{p'} \setminus \Upsilon_p $.}\\
Since the sets $\Upsilon_{p_i}$ are star-shaped with respect to the origin, 
it follows from~\eqref{W>} that the same is true for $\ul W_{p'}$. Moreover, in the present case, $\zeta \neq 0 \in (1 - \varepsilon) \Upsilon_p$, which implies that $\ul W_{p'}\neq\{0\}$. Thus, $\ul W_{p'}$ contains at least one set $\Upsilon_{p_i}\neq\{0\}$ in~\eqref{W>}. As a consequence, $\ul W_{p'}$ contains a neighbourhood of
the origin and therefore, as in the previous case, one can find a quantity
$\delta'>0$ such that \begin{equation*}
B_{\delta'}(\zeta)\subset (1-\eps/2)\ul W_{p'}.
\end{equation*}
By the inductive hypothesis, there exists a solution
	$u^{p'}$ with a compactly supported, continuous initial datum $0 \leq u_0^{p'} < \overline{p}$
	for which~\eqref{u>p} holds with $p$, $\ul W_p$ and $\eps$
	replaced by $p'$, $\ul W_{p'}$ and $\eps/2$ respectively.
	Then there exists $T'>0$ such~that
	\Fi{up'>0}
	\forall t\geq T',\ \forall x\in t 
	B_{\delta'}(\zeta)\subset (1-\varepsilon/2)  t\, \ul W_{p'},\quad
	u^{p'}(t,x)>p(x).
	\Ff
	
	\smallskip
{\em Case $(1-\varepsilon)^{-1} \zeta \in  \ul W_p \setminus ( \Upsilon_p \cup \ul W_{p'} )$.}\\
	The strategy is to decompose the time interval $(0,t)$
		into two parts where a solution spreads respectively over the shapes $\Upsilon_p$ and $\underline{W}_{p'}$.

			First, in this case there exist $\kappa \in (0,1)$, $z'\in \ul W_{p'} \setminus \Upsilon_p$ and $z\in  \Upsilon_p$ such that
			$$\zeta=(1-\eps)\kappa z' +(1-\eps)(1-\kappa)z.$$ 
		Consider the same function $u^{p'}$ as in the previous case. 
	We have shown that it fulfils~\eqref{up'>0} with $\zeta$ replaced by $(1-\eps)z'$, that we rewrite as follows:
	\Fi{up'>}
	\forall t\geq \frac {T'}\kappa,\ \forall x\in \kappa tB_{\delta'}\big((1-\eps)z'\big),\quad
	u^{p'}(\kappa t,x)>p(x).
	\Ff
	Consider now the function 
	$u^{p}$ of the first case. Take $\tau>0$ large enough so that 
	$$\supp u_0^p\subset B_{\delta'\kappa\tau -\sqrt{N}} (0).$$
	For $t\geq 0$, let $h_t\in\mathbb{Z}^N$ be such that 
	\Fi{ht}
	|h_t-(1-\eps)\kappa t z'|\leq\sqrt{N}.
	\Ff
	We deduce that, for $t\geq\tau$, there holds that  
	$$\forall x\notin\kappa tB_{\delta'}\big((1-\eps)z'\big),\quad
	u_0^p(x-h_t)=0,$$
	and therefore, if in addition $t\geq T'/\kappa$, then by \eqref{up'>} we have $u^{p'}(\kappa t,x)>u_0^p(x-h_t)$
	for all $x\in\R^N$.
	Thus, for $t\geq\max\{\tau, T'/\kappa\}$,
	we can apply the parabolic comparison principle to the functions
	$u^{p'}(\.+\kappa t,\.)$, $u^p(\.,\.-h_t)$, both satisfying~\eqref{eq:parabolic},
	and infer that
	$$\forall s>0,\ \forall x\in\R^N,\quad
	u^{p'}(s + \kappa t ,x)>u^p(s,x-h_t) .$$
	Hence, taking $s=(1-\kappa)t$, we obtain
	$u^{p'}(t,x+h_t)>u^p((1-\kappa)t,x)$ for all $x\in\R^N$.
	We now recall that for $u^p$, the convergence~\eqref{uptop} holds when~$\zeta=(1-\eps)z$.
	We deduce that
	$$\liminf_{t\to+\infty}\bigg(\inf_{x \in (1-\kappa)t B_\delta((1-\eps)z)} 
	\big(u^{p'}(t,x+h_t) - p (x)\big)\bigg) \geq 0.$$
	Recalling that $h_t$ satisfies~\eqref{ht}, we see that
	\begin{align*}
		(1-\kappa)t B_\delta((1-\eps)z)+\{h_t\}
	&=t B_{(1-\kappa)\delta}\Big((1-\eps)(1-\kappa)z+\frac{h_t}t\Big)\\
	&\supset t B_{(1-\kappa)\delta-\frac{\sqrt N}t}(\zeta),
	\end{align*}
	and therefore we conclude 
	\begin{equation}\label{eq:third_case}
	\liminf_{t\to+\infty}\bigg(\inf_{x \in t B_{(1-\kappa)\delta/2}(\zeta)} 
	\big(u^{p'}(t,x) - p (x)\big)\bigg) \geq 0.
	\end{equation}
	
	\smallskip
	{\em Conclusion}.\\
	According to either~\eqref{uptop}, \eqref{up'>0} or~\eqref{eq:third_case}, in each case we have found, for any $\zeta\in (1-\varepsilon)\ul W_p$, a solution $u^\zeta$ with 
	a compactly supported, continuous initial datum $0 \leq u_0^{\zeta} < \overline{p}$, and some $\delta_\zeta>0$
	such~that
	$$\liminf_{t\to+\infty}\bigg(\inf_{x \in t B_{\delta_\zeta}(\zeta)} 
	\big(u^\zeta(t,x) - p (x)\big)\bigg) \geq 0.$$
	Since $(1-\varepsilon)\ul W_p$ is compact, we can cover it by a finite number of 
	balls $B_{\delta_\zeta}(\zeta)$, $\zeta\in Z$ ($Z$ is a finite set). 
	Then we consider as an initial datum the function 
	$$u_0(x):=\max_{\zeta\in Z}u^\zeta(0,x),$$	
	which is continuous, compactly supported and satisfies $0 \leq u_0 < \ol p$.
	The solution with this initial datum satisfies the desired
	property~\eqref{u>p}, by the comparison principle.
\end{proof}

\subsection{Conclusion under the regularity Assumption~\ref{ass:tangentialC1} on the Wulff shapes}\label{sec:Wregular}

In the previous sections we have derived general upper and lower bounds,
Propositions~\ref{spread_upper} and~\ref{spread_lower_induc} respectively. 
They are expressed in terms of two different sets:  
the Wulff shape $W_{c[p]}$ of the speeds~$c[p]$ given by Definition~\ref{def:Wp},
and the set~$\underline{W}_{p}$, respectively, the latter being in turn 
defined recursively using the sets $W_{c_1^{p'}}$ and~$\Upsilon_{p'}$ (see Definition~\ref{defi:upsw}).
The role of Assumption~\ref{ass:tangentialC1} is to grant that 
these two bounds coincide and thus yield the sharp estimates of Theorem~\ref{th:spread_cpct11}.

The first step is to use the regularity of the Wulff shapes in Assumption~\ref{ass:tangentialC1} to relate the travelling front speeds with the supporting hyperplanes of the sets $W_{c_1^{p_k}}$. 
We recall that a supporting hyperplane for a closed, convex set is an hyperplane that intersects the set only at some boundary
points.
\begin{prop}\label{prop:wulff_reg}
	Let $e \in S^{N-1} \mapsto c(e)$ be a positive and lower semi-continuous function, and let~$W_c$ be the corresponding Wulff shape, i.e.
	$$W_c  = \bigcap_{e \in S^{N-1}} \{ x\in\R^N\ |\ x \cdot e \leq c (e) \}. $$
	If~$\partial W_c $ is a $C^1$-hypersurface, then for any $e \in S^{N-1}$, the hyperplane 
	$\{x\in\R^N\ |\  x \cdot e = c(e) \}$ is a supporting hyperplane of $W_c$, that is,
	$$\partial W_c  \cap \{ x \in \R^N \ | \ x \cdot e =c ( e) \} \neq \emptyset .$$
\end{prop}
As a matter of fact, in the proof of Theorem \ref{th:spread_cpct11}
we will not use the $C^1$ regularity of the Wulff shapes stated in Assumption \ref{ass:tangentialC1},
but only its consequence given by 
Proposition \ref{prop:wulff_reg}\,:
for any $k\in\{1,\dots,M\}$ such that the function $c_1^{p_{k-1}}$ is strictly positive (thus lower semi-continuous by Lemma~\ref{lem:lsc}),   
it holds that
\Fi{supphyp}
\forall e\in\Sph,\quad \
\{ x \in \R^N \ | \ x \cdot e  = c_1^{p_{k-1}}(e) \}\quad
\text{is a supporting hyperplane of }
W_{c_1^{p_{k-1}}}.
\Ff
We made the choice to state the stronger condition in Assumption \ref{ass:tangentialC1} only for the sake
of simplicity of presentation. 
We also point out that our counter-example in Section~\ref{sec:assumptions} works as such either way.

\begin{proof}[Proof of Proposition~\ref{prop:wulff_reg}]
	Being the set $W_c$ compact and convex, there exists $L_e\in\R$ such that $\{x\in\R^N\ |\ x \cdot e = L_e \}$ 
	is a supporting hyperplane to~$W_c$ at some point $\bar x\in \partial W_c$. 
	Let $(x_n)_{n\in\N}\subset\R^N\setminus W_c$ be a sequence
	converging to $\bar x$. By the definition of $W_c$, there exists $(e_n)_{n\in\N}$ in $\Sph$ such that 
	$x_n\. e_n>c(e_n)$. Letting $e'$ be the limit of (a subsequence of) $(e_n)_{n\in\N}$, the lower semi-continuity
	of~$c$ yields $\bar x\. e' \geq c(e')$. Observe that the reverse inequality also holds, because $\bar x\in W_c$.
	We conclude that $\bar x \cdot e' = c(e')$.
	This shows that $\{ x\in\R^N\ |\ x \cdot e'=c (e') \}$ is also a supporting 
	hyperplane to~$W_c$ at~$\bar x $. As a consequence, by the uniqueness of the tangential hyperplane to 
	the $C^1$-hypersurface $W_c$, we conclude that $e' = e$. This ends the proof.
\end{proof}

As explained above, our goal is to show that
\Fi{WpkWk}
\forall k \in \{1, \dots , M\},\quad
\underline{W}_{p_{k-1}} = W_{c[p_{k}]},
\Ff
where, by Assumption~\ref{ass:tangentialC1}, the $p_0, \cdots, p_M$ are all the (ordered) stable, periodic steady states. 
The proof relies on the following result.

\begin{prop}\label{prop:coincide}
	Under Assumption~\ref{ass:tangentialC1}, for any $k \in \{1, \dots, M\}$ and for any $e \in S^{N-1}$, 
	the hyperplane $\{ x \in \R^N \ | \ x \cdot e = c [p_{k}] (e) \}$ is a supporting hyperplane to $\underline{W}_{p_{k-1}}$.
\end{prop}

Let us postpone for a moment the proof of Proposition \ref{prop:coincide}, and show how it yields
Theorem~\ref{th:spread_cpct11}.	As we announced above, we will prove this proposition, and thus Theorem~\ref{th:spread_cpct11}, in the case where the~$C^1$ regularity of $W_{c_1^{p_k}}$ in Assumption \ref{ass:tangentialC1} 
is relaxed by \eqref{supphyp}.


	\begin{proof}[Proof of Theorem~\ref{th:spread_cpct11}]
		For any solution~$u$ with a compactly supported initial datum $0 \leq u_0 \leq \overline{p}$ 
for which~\eqref{spread_start} holds, Propositions~\ref{spread_upper} and~\ref{spread_lower_induc} provide us with the following estimates:
for any $\eps\in(0,1)$ and any $k \in \{1,\dots, M\}$, there holds
$$
 \limsup_{t \to +\infty}
\bigg( \sup_{x   \in\, \R^N\setminus   (1+\varepsilon)t  W_{c[p_k]}} \big(u(t,x) -  p_k (x)\big)\bigg) \leq 0,
$$
$$
\liminf_{t \to +\infty}
\bigg( \inf_{x  \in  (1- \varepsilon)t   \underline{W}_{p_{k-1}}  } \big(u(t,x) -  p_{k-1} (x)\big)\bigg) \geq 0.
$$
%
%
In order to deduce from 
the above estimates the desired ones, we need to prove \eqref{WpkWk}.


The inclusion
$$\underline{W}_{p_{k-1}} \subset W_{c[p_{k}]}$$
is already granted by the above estimates, 
since $p_{k-1}>p_{k}$.
Let us turn to the reverse inclusion.
Owing to Definition~\ref{defi:upsw} and the characterisation \eqref{w_p_conv} in Lemma~\ref{lem:ulW}, 
we know that the set $\underline{W}_{p_{k-1}}$ is compact and convex,
therefore it can be written as
\Fi{Wpk}
\underline{W}_{p_{k-1}}=\bigcap_{e \in \Sph}\{ x \in \R^N \ | \  x \cdot e \leq L_e \},
\Ff
for some quantities $(L_e)_{e\in\Sph}$. It then follows from Proposition \ref{prop:coincide}
that $L_e\geq c [p_{k}] (e)$ for all~$e \in \Sph$. This means that $\underline{W}_{p_{k-1}} \supset W_{ c [p_{k} ]}$.
\end{proof}

%
%

In conclusion, it only remains to prove Proposition~\ref{prop:coincide}.
For convenience, let us again recall, on the one hand, that by Definition~\ref{def:Wp},  $c [p_k] (e)$ is the speed of the unique travelling front contained in the terrace connecting~$\ol p$ to~$0$ in direction~$e$ whose profile fulfils~\eqref{eq:through} with $p =p_k$. On the other hand, from 
Definition~\ref{def:Wp} $(ii)$, $c_1^{p_k} (e)$ is the speed of the uppermost travelling front of the terrace connecting~$p_k$ to~$0$ in direction~$e$. It is then straightforward to check the following properties, that will come 
in handy in the proof of Proposition~\ref{prop:coincide}.
\begin{lem}\label{lem_coincide}
Under Assumption~\ref{ass:tangentialC1}, for any $k \in \{1,\dots,M-1\}$, it holds:
\begin{enumerate}[$(i)$]
\item $ c [p_k] (e) \leq c [p_{k+1}] (e)$;
\item if $p_k$ is a platform of the terrace connecting~$\ol p$ to~$0$ in direction~$e$, then 
$$c [p_{k+1}] (e) = c_1^{p_k} (e);$$
\item if $p_k$ is not a platform of the terrace connecting~$\ol p$ to~$0$ in direction~$e$, then 
$$c [p_{k+1}] (e) = c [p_k] (e).$$
\end{enumerate}
\end{lem}

Property~$(i)$ simply follows from the ordering of speeds in the definition of a terrace. 
Indeed, even though $p_k$ may not be one of the platforms of the terrace $\mathcal{T}^e$
connecting~$\ol p$ to~$0$ in direction~$e$, the quantity $ c [p_k] (e)$
coincides with the speed of one of its fronts, and $c [p_{k+1}] (e)$ with the speed of the same front or of a lower one,
hence $c [p_k] (e)\leq c [p_{k+1}] (e)$.
Concerning $(ii)$, when $p_k$ is a platform of the original terrace $\mathcal{T}^e$ connecting~$\overline{p}$ to~$0$, 
then the terrace connecting~$p_k$ to~$0$ is contained in $\mathcal{T}^e$. In particular $c_1^{p_k} (e)$ is also the speed of the travelling front
belonging to~$\mathcal{T}^e$, which connects~$p_k$ to 
some~$p_i$ with $i \geq k+1$. In other words, $c_1^{p_k} (e) = c [p_{k+1}] (e)$ as announced. For property~$(iii)$, if~$p_k$ is not a platform of~$\mathcal{T}^e$, then $\mathcal{T}^e$ contains a travelling front which connects some~$p_j$ to some other~$p_i$, with $j < k < i$. It follows that the profile of that travelling front 
 fulfils~\eqref{eq:through} for both $p_k$ and $p_{k+1}$, hence $c [p_{k+1}] (e) =c [p_k] (e)$.

\begin{proof}[Proof of Proposition~\ref{prop:coincide}]
We proceed by induction. In the case $k = 1$ we have by Definition~\ref{defi:upsw} $\underline{W}_{p_0} = \Upsilon_{p_0} =W_{c_1}^{p_0}$,
the last equality being a consequence of the fact that $c_1^{p_0}=c_1^{\overline{p}}= c [p_1]>0$
thanks to Assumption~\ref{ass:tangentialC1}. The result then follows from~\eqref{supphyp}. 

Now assume that the conclusion holds for some $k \geq 1$, and let us show it holds true for~$k+1$. 
We make use of the characterisation~\eqref{w_p_conv}, which holds by Lemma \ref{lem:ulW}
because the stable steady states are ordered, due to Assumption~\ref{ass:tangentialC1}.
Notice that~\eqref{w_p_conv} can be rewritten as
\begin{equation}\label{W>recursive_new}
\underline{W}_{p_k} = \text{Conv} \,  ( \underline{W}_{p_{k-1} } \cup \Upsilon_{p_k}) .
\end{equation}
It follows from~\eqref{w_p_conv} that $\underline{W}_{p_k}$ is convex and compact, 
hence
we can further rewrite it as in~\eqref{Wpk}, assuming in addition that
the hyperplanes 
$$P_e := \{ x \in \R^N \ | \ x \cdot e = L_e \}$$
are supporting hyperplanes for~$\underline{W}_{p_k}$.
Since $\underline{W}_{p_k}\supset\Upsilon_{p_0}$,
which contains a neighbourhood of the origin due to Assumption~\ref{ass:tangentialC1}
and the lower semi-continuity of $e \mapsto c_1^{\ol p} (e)$ (see also Lemma~\ref{lem:lsc}),
one has that $L_e>0$ for all $e \in S^{N-1}$.

Our goal is to show that, for any $e \in \Sph$,
\begin{equation*}\label{eq:goal12}
L_e = c [p_{k+1}] (e).\end{equation*}
%
First, due to \eqref{W>recursive_new} the hyperplane $P_e$ is a supporting hyperplane of either $\underline{W}_{p_{k-1}}$ or $\Upsilon_{p_k}$. Notice that the latter case is only possible if $c_1^{p_k} >0$ 
and $\Upsilon_{p_k} =W_{c_1^{p_k}}$ (otherwise $\Upsilon_{p_k}$ would reduce to the singleton $\{0\}$ by Definition~\ref{defi:upsw}, while $L_e >0$). From our inductive hypothesis and~\eqref{supphyp},
 we then~get
$$L_e  = \max\{c [p_k] (e),c_1^{p_k} (e)\}.$$
Let us distinguish two cases.

Suppose first that  $p_k$ is a platform of the terrace $\mathcal{T}^e$ connecting~$\overline{p}$ to~$0$ in direction $e$.
Then statements~$(i)$ and $(ii)$ of Lemma~\ref{lem_coincide} yield
$$c [p_k] (e)\leq c [p_{k+1}] (e)=c_1^{p_k} (e),$$
whence $L_e =\max\{c [p_k] (e),c_1^{p_k} (e)\}= c [p_{k+1}] (e)$ in this case, which is the desired identity.

Suppose now that $p_k$ is not a platform of the terrace $\mathcal{T}^e$.
Then Lemma~\ref{lem_coincide} $(iii)$ yields
$$c [p_{k+1}] (e)=c [p_k] (e).$$
Moreover,  $\mathcal{T}^e$ contains a travelling front connecting~$p_j$ to~$p_i$ with $j <k < i$ and its speed is equal to~$c [p_{k}] (e)$.
Our spreading result in the planar-like case, Theorem~\ref{th:planar_speeds},
implies, on the one hand, that
the solution~$u_1$ starting from the initial datum
$$\1_{(-\infty,0]} (x \cdot e) \times \overline{p}(x)$$
satisfies
$$\forall \varepsilon >0,\quad
\limsup_{t \to +\infty} \left( \sup_{x \cdot e \geq (c [p_k] (e) + \varepsilon ) t} (u_1 (t,x) - p_i(x)) \right) \leq 0.$$
On the other hand, applying Theorem~\ref{th:planar_speeds} to the terrace
$\mathcal{T}^e_{p_k}$ connecting~$p_k$ to~$0$ in direction $e$, 
we infer that the solution $u_2$ starting from the initial datum
$$\1_{(-\infty,0]} (x \cdot e) \times p_k(x)$$
satisfies
$$\forall \varepsilon >0,\quad
\liminf_{t \to +\infty} \left( \inf_{x \cdot e \leq (c_1^{p_k} (e) - \varepsilon )t } ( u_2 (t,x) - p_k (x) ) \right) \geq 0.$$
Therefore, since
$u_1 (t,x ) \geq u_2 (t,x)$ for all $t>0$ and $x \in \R^N$,  by the comparison principle, and $p_i<p_k $, 
we deduce by the arbitrariness of $\eps>0$, that
$$c_1^{p_k}(e) \leq  c [p_k] (e).$$
This shows that
$$\max\{c [p_k] (e),c_1^{p_k} (e)\}=c [p_k] (e),$$
and thus eventually $L_e=c [p_k] (e)=c [p_{k+1}] (e)$.

We have thereby proved that the conclusion holds for $k+1$. 
This ends the proof of Proposition~\ref{prop:coincide}, hence of Theorem~\ref{th:spread_cpct11}.
\end{proof}

\section{Discussion about our assumptions}\label{sec:assumptions}

Our results for compactly supported initial data require either Assumption~\ref{ass:same_shape} or Assumption~\ref{ass:tangentialC1}.
We now describe some situations where these assumptions may fail.

As we have discussed in the introduction, we have exhibited in our previous work~\cite{GR_2020} 
	an equation for which there exist some directions in which the propagating terrace is composed by two fronts,
		 and others in which it is composed by just one front. Therefore Assumption~\ref{ass:same_shape} does not hold in that case. In fact, our construction in~\cite{GR_2020} shares some similarities with the one below, and in particular also relied on the fact that bistable speeds may vary with the direction of the fronts, eventually resulting in different terrace shapes according to the procedure described in Section~\ref{sec:algorithm} here. We will not explore further that aspect and instead focus on our other assumptions.

While one may expect the regularity of the Wulff shapes 
to hold true in most situations, still there exist some counter-examples to Assumption~\ref{ass:tangentialC1}. 
These counter-examples can be constructed by building on the results of~\cite{DingGiletti}, which roughly state that the travelling front speeds in different directions may not only differ but are mostly independent of each other, apart from the fact they must have the same sign.
 
More precisely, the following theorem will yield the existence of a bistable equation
for which the Wulff shape is not smooth and therefore Assumption~\ref{ass:tangentialC1} fails, 
as well as the weaker property~\eqref{supphyp}.
\begin{theo}[\cite{DingGiletti}, Theorem 3]\label{thm:DG}
	In dimension $N\geq2$, for any finite set of directions with rational coordinates
	$e_1, \dots , e_I \in S^{N-1} \cap \mathbb{Q}^{N}$ with $e_i \neq e_j$, and any $c_1 ,\dots, c_I \geq 0$, there exists a spatially periodic and bistable equation of the type
	$$\partial_t u  = \Delta u + f(x,u),$$
	such that, for all $i\in\{1,\dots I\}$,
	 the (unique) speed of pulsating travelling fronts in direction $e_i$
	is equal to $c_i$.
\end{theo}
\begin{rmk}\label{rmk:trustmebro}
By bistable, it is meant here, as well as in~\cite{DingGiletti}, that~$0$ and~$1$ are linearly stable periodic steady states, and that any other periodic steady state is linearly unstable. As far as periodicity is concerned, up to some rescaling and without loss of generality we can always assume that the function~$f$ has period~$1$ in each direction of the canonical basis.

We should also point out that the proof of~\cite[Theorem 3]{DingGiletti} 
relies on the construction of the reaction term as 
$$f(x,u) = u (1-u) (u-1/2) + \chi (x,u).$$
Here $u \mapsto u(1-u) (u-1/2)$ is the balanced bistable nonlinearity, whose corresponding reaction-diffusion equation admits a travelling front~$U_0 (x \cdot e)$ with zero speed in any direction~$e \in S^{N-1}$.  On the other hand, $\chi$ is a well chosen smooth, nonnegative and nontrivial spatially periodic function; see~\cite[Section 4.1]{DingGiletti}. The nonnegativity of $\chi$ will actually play an important role in the construction of our counter-example to Assumption~\ref{ass:tangentialC1}.
\end{rmk}


By suitably choosing a set of directions and a set of speeds, one can construct a 
Wulff shape which exhibits a ``corner'' and thus violates Assumption~\ref{ass:tangentialC1}. We recall that, if $(c^*(e))_{e\in\Sph}$ is the set of speeds of the pulsating travelling fronts in the corresponding directions,
their Wulff shape is 
$$W^* := W_{c^*}= \bigcap_{e \in S^{N-1}} \{ x \in \R^N \ | \ x \cdot e \leq c^* (e) \} .$$

\begin{cor}\label{ce:C1}
In any dimension $N\geq2$, there exists a bistable and spatially periodic equation of the type
	$$\partial_t u  = \Delta u + f(x,u),$$
	with $0$ and $1$ being the linearly stable steady states, 
	such that the Wulff shape $W^*$ of the speeds of the travelling fronts
	is nontrivial (i.e.~it has non-empty interior), yet its boundary is not of class $C^1$ and 
	moreover one of the hyperplanes $\{ x \in \R^N \ | \ x \cdot e = c^* (e) \}$ is not a supporting hyperplane. 
\end{cor}
 
\begin{proof}
	\Step{1}{Using Theorem~\ref{thm:DG}.}
	We take two orthogonal directions $e_1,e_2\in\Sph$ and call
	$$\hat e:=\frac 35e_1+\frac45e_2\in\Sph.$$
	We then apply Theorem~\ref{thm:DG} and find a bistable, spatially periodic equation
	\begin{equation}\label{eq:tildef}
	\partial_t u  = \Delta u + \tilde{f} (x,u),
	\end{equation}
	such that the corresponding speeds $\tilde{c} (e_1)$, $\tilde{c} (e_2)$, $\tilde{c} (\hat e)$
	of the travelling fronts in the directions $e_1$, $e_2$, and~$\hat e$ satisfy
	\begin{equation}\label{eq:choice_speeds}
	\tilde{c} (e_1)=\tilde{c} (e_2)=1,\qquad \tilde{c} (\hat e)=2.
	\end{equation}
However, the main difficulty now  is that Theorem~\ref{thm:DG}, as stated in~\cite{DingGiletti}, does not ensure the positivity of $\tilde{c}$ in other directions. We only know, by a simple integration argument on the profile's equation (see e.g.~\cite{DHZ17,Ducrot2}), that the front speed in any given direction must have the same sign, in the large sense, as
the quantity
$$\int_{0}^1 \int_{(0,1)^N} \tilde{f} (x,u) dx du .$$
In particular, here $\tilde{c} (e) \geq 0$ for any $e \in S^{N-1}$. Still, the Wulff shape may reduce to the origin 
and spreading may not occur. The remainder of the proof will be mostly devoted to a perturbation of~\eqref{eq:tildef} to overcome this issue. If~$\tilde{c}$ were known to be positive in all directions, one may directly skip to the last step of this proof.

\Step{2}{Ensuring the positivity of the speeds.}
According to Remark~\ref{rmk:trustmebro}, we may also assume without loss of generality that, for any $x \in \mathbb{R}^N$ and $u \in [0,1]$,
	$$\tilde{f} (x,u) \geq u (1-u) (u-1/2).$$
	Let us consider, for $\eps>0$, 
	$$f_\varepsilon (x,u) = \tilde{f} (x,u) + \varepsilon u (1-u).$$
	One may check that the corresponding parabolic equation
	\begin{equation}\label{eq:justef}
	\partial_t u = \Delta u + f_\varepsilon (x,u)
	\end{equation}
	remains of the bistable type for $\varepsilon >0$ small enough. This relies on the continuity of the periodic principal eigenvalue~$\lambda  (\sigma)$ of the operator
			$$\mathcal{L}_\sigma : \varphi \mapsto \Delta \varphi + \sigma (x) \varphi,$$
		with respect to $\sigma \in L^\infty ([0,1]^N; \mathbb{R})$; see e.g.~\cite{Kato}. It indeed 
		yields that, for $p\equiv0$ or $p\equiv1$,
		then $\lambda (\partial_u f_\varepsilon (\cdot,p)) \to  \lambda (\partial_u \tilde{f} (\cdot,p)) < 0$ 
		as $\varepsilon \to 0$, hence $0$ and $1$ are linearly stable solutions to~\eqref{eq:justef}
		 for any~$\varepsilon$ small enough.

Moreover, assume by contradiction that there exist $\varepsilon_n \to 0$ and periodic steady states~$p_n \in (0,1)$ of~\eqref{eq:justef}, with $\varepsilon = \varepsilon_n$, which are not linearly unstable, that is $\lambda (\partial_u f_{\varepsilon_n}  (\cdot, p_n)) \leq 0$. Then, by standard compactness arguments, the sequence $p_n$ converges uniformly as $n \to +\infty$ to some periodic steady state~$p \in [0,1]$ of~\eqref{eq:tildef}. Again by continuity of the periodic principal eigenvalue, we get that $\lambda (\partial_u \tilde{f} (\cdot, p)) \leq 0$. Due to~\eqref{eq:tildef} being bistable, we infer that either~$p \equiv 0$ or~$p \equiv 1$, hence
$\lambda ( \partial_u \tilde{f} (\cdot,p))<0$. In the former case, consider
$$\overline{u} (t,x) := \delta e^{ \frac{\lambda ( \partial_u \tilde{f} (\cdot,0))}{2} t }  \varphi (x),$$
where $\delta>0$ and~$\varphi$ is a positive eigenfunction of $\mathcal{L}_{\partial_u \tilde{f} (\cdot,0)}$.
The negativity of $\lambda ( \partial_u \tilde{f} (\cdot,0))$ implies that $\ol u(t,x)\to0$ as $t\to+\infty$,
and moreover that there exists $\eta >0$ such that, if $\delta, \varepsilon \in (0, \eta)$, 
then~$\overline{u}$ is a supersolution of~\eqref{eq:justef}. In particular, taking $\delta \in (0,\eta)$ and $n$ large enough such that $\varepsilon_n \in (0,\eta)$ and $0 < \| p_n \|_\infty \leq \delta \min \varphi $, we get by the parabolic comparison principle that $p_n \leq \overline{u}$, a contradiction with the fact that $\ol u$
tends to $0$ as $t\to +\infty$. The case $p\equiv 1$ can be handled similarly.

Now that we know that the equation~\eqref{eq:justef} is bistable, provided that $\eps$ is sufficiently small,
we deduce from \cite{GR_2020} the existence of a (bistable) pulsating travelling front of~\eqref{eq:justef} with a unique speed~$c^*_\varepsilon (e)$ in any direction~$e \in S^{N-1}$. Let us show now that 
\Fi{ceps>0}
\forall e\in\Sph,\quad c^*_\eps(e)>0.
\Ff
We have that
$$f_\eps(u)\geq u(1-u)(u-\frac12+\eps)=:g_\eps(u).$$
It is a classical result that the solution $u$ of the equation $\partial_t u = \Delta u + g_\varepsilon (u) $,
with a compactly supported initial datum larger than $1/2$ on a suitably large ball, 
satisfies the invasion property, i.e.~$u(t,x)\to1$ as $t\to+\infty$,
locally uniformly in $x\in\R^N$, see \cite{AW}.
Therefore, the comparison principle implies
that this holds true for the solution of the equation~\eqref{eq:justef}
with the same initial datum, which in turn immediately yields \eqref{ceps>0}, again by comparison.

Finally, we claim that
\begin{equation}\label{claim:continuity_speeds}
c^*_\varepsilon (e)\searrow \tilde{c} (e) \quad\text{as }\;\varepsilon \searrow 0 .
\end{equation}
To verify~\eqref{claim:continuity_speeds}, let us fix $e \in S^{N-1}$. Notice first that, for any $0 < \varepsilon_1 < \varepsilon_2$, we have $\tilde{f} \leq f_{\varepsilon_1} \leq f_{\varepsilon_2}$. Due to the spreading result Theorem~\ref{th:planar_speeds} and the parabolic comparison principle, necessarily $\tilde{c} (e) \leq c^*_{\varepsilon_1} (e) \leq c^*_{\varepsilon_2} (e)$. In particular $c^*_\eps (e)$ admits a limit $c^*_0(e)$ as $\eps\searrow0$, which fulfils
$$
c^*_\varepsilon (e)\searrow c^*_0(e)\geq\tilde{c} (e)\geq0 \quad\text{as }\;\varepsilon \searrow 0.$$
If $c^*_0 (e) = 0$, then immediately~\eqref{claim:continuity_speeds} holds true. Assume instead that $c^*_0 (e) >0$. 
Then for $\eps>0$ sufficiently small, consider the bistable pulsating front of~\eqref{eq:justef} in direction~$e$, 
	$$u_\varepsilon (t,x) = U_\varepsilon (x, x \cdot e - c^*_\varepsilon (e) t),$$
shifted so that $$\max U_\varepsilon (\cdot, 0) = \eta,$$
and $\eta >0$ belongs to the basin of attraction of~$0$ with respect to~\eqref{eq:tildef}. Recall that $U_\varepsilon$ is nonincreasing with respect to its second variable according to Theorem~\ref{th:existence},
hence $u_\eps$ is nondecreasing in $t$ due to~\eqref{ceps>0}.
By standard parabolic estimates, up to extraction of a subsequence we have that 
$$\lim_{\varepsilon \searrow 0} u_\varepsilon = \tilde{u},$$
where the convergence is understood in the locally uniform sense and also holds for the first order in 
time and second order in space derivatives. Then $\tilde{u}$ satisfies~\eqref{eq:tildef}, as well 
as~$\partial_t \tilde u \geq 0$ and
\begin{equation}\label{eq:almost_front}
\tilde{u} (t+T,x +L) = \tilde{u} (t,x),
\end{equation}
for any $L\in \mathbb{Z}^N$ and $T = \frac{L \cdot e}{c^*_0 (e)}$. For $\tilde{u}$ to be a bistable pulsating travelling front of~\eqref{eq:tildef} in direction~$e$, it only remains to check that $\tilde{u} $ converges to~$0$ (resp.~$1$) as $t\to -\infty$ (resp.~$t \to +\infty$). The former convergence swiftly follows from our choice of~$\eta$ in the basin of attraction of~$0$. For the latter, notice first that, by its time monotonicity, the function~$\tilde{u}$ converges as $t\to + \infty$ to some steady state~$q>0$, which is also spatially periodic by passing to the limit in~\eqref{eq:almost_front}. 
The function $q$ cannot be a linearly unstable steady state, because then Lemma~\ref{lem:counter} 
would imply that the speed $c^*_0(e)$ of
the front~$\tilde{u}$ connecting~$q$ and~$0$ should have a negative speed, which is not the case.
Thus $q\equiv 1$, and therefore we conclude that $\tilde{u}$ is a bistable pulsating travelling front with speed $c_0^* (e)\geq0$, hence by 
our uniqueness result, Theorem \ref{th:uniqueness}, 
$c_0^* (e) = \tilde{c} (e)$. This ends the proof of~\eqref{claim:continuity_speeds}.


\Step{3}{Conclusion of the proof.}
 We know that the Wulff shape $W^*$ of the speeds of the travelling fronts
	for the equation~\eqref{eq:justef} is a compact, convex set, which by~\eqref{ceps>0}
	 and the lower semi-continuity of $e \mapsto c^*_\varepsilon (e)$~\cite{Guo_speeds,RossiFG}  also contains an open neighborhood of~$0$. 
	 
	By direct inspection, one derives from \eqref{eq:choice_speeds} that there exists $\delta>0$ such that
$$\{ x \in \R^N \ | \ x \cdot e_1 \leq \tilde{c} (e_1)+\delta \}\cap 
	\{ x \in \R^N \ | \ x \cdot e_2 \leq \tilde{c} (e_2)+\delta \}\subset
	\{ x \in \R^N \ | \ x \cdot \hat e <  \tilde{c} (\hat e) \}$$
Therefore, thanks to~\eqref{claim:continuity_speeds}, one infers that
	$$W^*\subset \{ x \in \R^N \ | \ x \cdot e_1 \leq c^*_\varepsilon (e_1) \}\cap 
	\{ x \in \R^N \ | \ x \cdot e_2 \leq c^*_\varepsilon (e_2) \}\subset
	\{ x \in \R^N \ | \ x \cdot \hat e <  c^*_\varepsilon (\hat e) \},$$
provided $\varepsilon$ is small enough. 
This shows that the hyperplane $\{ x \in \R^N \ | \ x \cdot \hat e= c^*_\varepsilon (\hat e)\}$ is not a supporting hyperplane of~$W^*$, whence in particular Assumption~\ref{ass:tangentialC1} is violated,
 owing to Proposition~\ref{prop:wulff_reg}.
\end{proof}


\section*{Acknowledgements}

This work has received funding from the French ANR projects: ReaCh (ANR-23-CE40-0023-02), Indyana (ANR-21-CE40-0008), and from the European Union -- Next Generation EU, PRIN project: 2022W58BJ5 ``PDEs and optimal control methods in mean field games, population dynamics and multi-agent models''.


\bibliographystyle{abbrv}
\bibliography{GR_Bibliography}


\end{document}